\documentclass[11pt]{article} 
	\usepackage{caption}
    \usepackage{url}
    \usepackage{verbatim}
	\usepackage{graphicx}
	\usepackage{amsmath}
	\usepackage{amsthm}
	\usepackage{amsfonts}
	\usepackage{nicefrac}
	\usepackage{longtable}
	\usepackage{caption}
	\usepackage{graphicx, amssymb, subcaption,graphics}
	\usepackage{tikz}
	\usepackage{pgfplots}
	\usetikzlibrary{spy,calc}
	\usepackage{hyperref}
	\allowdisplaybreaks
	    
	\newcommand{\mD}{\mathfrak D}
	\newcommand{\md}{\mathfrak d}

	\newcommand{\nabh}{\nabla_{\! h}}

	\newcommand{\hf}{\nicefrac{1}{2}}
	\newcommand{\nrm}[1]{\left\| #1 \right\|}
	\newcommand{\ciptwo}[2]{\left( #1 , #2 \right)}

	\newcommand{\eipns}[2]{\left[ #1 , #2 \right]_{\rm ns}}
	\newcommand{\eipew}[2]{\left[ #1 , #2 \right]_{\rm ew}}
	\newcommand{\viptwo}[2]{\left\langle #1, #2 \right\rangle}
	\newcommand{\monenrm}[1]{\left\| #1 \right\|_{-1}}
	\newcommand{\moneinn}[1]{\left( #1 \right)_{-1}}
	\newcommand{\msfT}{\mathsf{T}}
	
	\def\x{\mbox{\boldmath $x$}}

	\newcommand{\iprd}[2]{\left( #1 , #2 \right)}
    \def\norm#1#2{\left\| #1 \right\|_{#2}}
	    	
	\newtheorem{thm}{Theorem}[section]
	\newtheorem{prop}[thm]{Proposition}
	\newtheorem{cor}[thm]{Corollary}

	\newtheorem{lem}[thm]{Lemma}
	\newtheorem{rem}[thm]{Remark}

	\newif\ifblackandwhitecycle
	\gdef\patternnumber{0}
	
	\pgfkeys{/tikz/.cd,
	    zoombox paths/.style={
	        draw=orange,
	        very thick
	    },
	    black and white/.is choice,
	    black and white/.default=static,
	    black and white/static/.style={ 
	    		        draw=white,   
       				zoombox paths/.append style={
		            draw=white,
		            postaction={
		   	            draw=black,
		                loosely dashed
		            }
		        }
		    },
		    black and white/static/.code={
		        \gdef\patternnumber{1}
		    },
		    black and white/cycle/.code={
		        \blackandwhitecycletrue
		        \gdef\patternnumber{1}
		    },
		    black and white pattern/.is choice,
		    black and white pattern/0/.style={},
		    black and white pattern/1/.style={    
		            draw=white,
		            postaction={
		                draw=black,
		                dash pattern=on 2pt off 2pt
		            }
		    },
		    black and white pattern/2/.style={    
		            draw=white,
		            postaction={
		                draw=black,
		                dash pattern=on 4pt off 4pt
		            }
		    },
		    black and white pattern/3/.style={    
		            draw=white,
		            postaction={
		                draw=black,
		                dash pattern=on 4pt off 4pt on 1pt off 4pt
		            }
		    },
		    black and white pattern/4/.style={    
		            draw=white,
		            postaction={
		                draw=black,
		                dash pattern=on 4pt off 2pt on 2 pt off 2pt on 2 pt off 2pt
		            }
		    },
		    zoomboxarray inner gap/.initial=5pt,
		    zoomboxarray columns/.initial=2,
		    zoomboxarray rows/.initial=2,
		    subfigurename/.initial={},
		    figurename/.initial={zoombox},
		    zoomboxarray/.style={
		        execute at begin picture={
		            \begin{scope}[
		                spy using outlines={%
		                    zoombox paths,
		                    width=\imagewidth / \pgfkeysvalueof{/tikz/zoomboxarray columns} - (\pgfkeysvalueof{/tikz/zoomboxarray columns} - 1) / \pgfkeysvalueof{/tikz/zoomboxarray columns} * \pgfkeysvalueof{/tikz/zoomboxarray inner gap} -\pgflinewidth,
		                    height=\imageheight / \pgfkeysvalueof{/tikz/zoomboxarray rows} - (\pgfkeysvalueof{/tikz/zoomboxarray rows} - 1) / \pgfkeysvalueof{/tikz/zoomboxarray rows} * \pgfkeysvalueof{/tikz/zoomboxarray inner gap}-\pgflinewidth,
		                    magnification=3,
		                    every spy on node/.style={
		                        zoombox paths
		                    },
		                    every spy in node/.style={
		                        zoombox paths
		                    }
		                }
		            ]
		        },
		        execute at end picture={
		          \end{scope}
		           
		     \gdef\patternnumber{0}
		        },
		        spymargin/.initial=0.5em,
		        zoomboxes xshift/.initial=1,
		        zoomboxes right/.code=\pgfkeys{/tikz/zoomboxes xshift=1},
		        zoomboxes left/.code=\pgfkeys{/tikz/zoomboxes xshift=-1},
		        zoomboxes yshift/.initial=0,
		        zoomboxes above/.code={
		            \pgfkeys{/tikz/zoomboxes yshift=1},
		            \pgfkeys{/tikz/zoomboxes xshift=0}
		        },
		        zoomboxes below/.code={
		            \pgfkeys{/tikz/zoomboxes yshift=-1},
		            \pgfkeys{/tikz/zoomboxes xshift=0}
		        },
		        caption margin/.initial=4ex,
		    },
		    adjust caption spacing/.code={},
		    image container/.style={
		        inner sep=0pt,
		        at=(image.north),
		        anchor=north,
		        adjust caption spacing
		    },
		    zoomboxes container/.style={
		        inner sep=0pt,
		        at=(image.north),
		        anchor=north,
		        name=zoomboxes container,
		        xshift=\pgfkeysvalueof{/tikz/zoomboxes xshift}*(\imagewidth+\pgfkeysvalueof{/tikz/spymargin}),
		        yshift=\pgfkeysvalueof{/tikz/zoomboxes yshift}*(\imageheight+\pgfkeysvalueof{/tikz/spymargin}+\pgfkeysvalueof{/tikz/caption margin}),
		        adjust caption spacing
		    },
		    calculate dimensions/.code={
		        \pgfpointdiff{\pgfpointanchor{image}{south west} }{\pgfpointanchor{image}{north east} }
		        \pgfgetlastxy{\imagewidth}{\imageheight}
		        \global\let\imagewidth=\imagewidth
		        \global\let\imageheight=\imageheight
		        \gdef\columncount{1}
		        \gdef\rowcount{1}
		        
		    },
		    image node/.style={
		        inner sep=0pt,
		        name=image,
		        anchor=south west,
		        append after command={
		            [calculate dimensions]
		            node [image container,subfigurename=\pgfkeysvalueof{/tikz/figurename}-image] {\phantomimage}
		            node [zoomboxes container,subfigurename=\pgfkeysvalueof{/tikz/figurename}-zoom] {\phantomimage}
		        }
		    },
		    color code/.style={
		        zoombox paths/.append style={draw=#1}
		    },
		    connect zoomboxes/.style={
		    spy connection path={\draw[draw=none,zoombox paths] (tikzspyonnode) -- (tikzspyinnode);}
		    },
		    help grid code/.code={
		        \begin{scope}[
		                x={(image.south east)},
		                y={(image.north west)},
		                font=\footnotesize,
		                help lines,
		                overlay
		            ]
		            \foreach \x in {0,1,...,9} { 
		                \draw(\x/10,0) -- (\x/10,1);
		                \node [anchor=north] at (\x/10,0) {0.\x};
		            }
		            \foreach \y in {0,1,...,9} {
		                \draw(0,\y/10) -- (1,\y/10);                        \node [anchor=east] at (0,\y/10) {0.\y};
		            }
		        \end{scope}    
		    },
		    help grid/.style={
		        append after command={
		            [help grid code]
		        }
		    },
		}
		
		\newcommand\phantomimage{%
		    \phantom{%
		        \rule{\imagewidth}{\imageheight}%
		    }%
		}
		\newcommand\zoombox[2][]{
		    \begin{scope}[zoombox paths]
		        \pgfmathsetmacro\xpos{
		            (\columncount-1)*(\imagewidth / \pgfkeysvalueof{/tikz/zoomboxarray columns} + \pgfkeysvalueof{/tikz/zoomboxarray inner gap} / \pgfkeysvalueof{/tikz/zoomboxarray columns} ) + \pgflinewidth
		        }
		        \pgfmathsetmacro\ypos{
		            (\rowcount-1)*( \imageheight / \pgfkeysvalueof{/tikz/zoomboxarray rows} + \pgfkeysvalueof{/tikz/zoomboxarray inner gap} / \pgfkeysvalueof{/tikz/zoomboxarray rows} ) + 0.5*\pgflinewidth
		        }
		        \edef\dospy{\noexpand\spy [
		            #1,
		            zoombox paths/.append style={
		                black and white pattern=\patternnumber
		            },
		            every spy on node/.append style={#1},
		            x=\imagewidth,
		            y=\imageheight
		        ] on (#2) in node [anchor=north west] at ($(zoomboxes container.north west)+(\xpos pt,-\ypos pt)$);}
		        \dospy
		        \pgfmathtruncatemacro\pgfmathresult{ifthenelse(\columncount==\pgfkeysvalueof{/tikz/zoomboxarray columns},\rowcount+1,\rowcount)}
		        \global\let\rowcount=\pgfmathresult
		        \pgfmathtruncatemacro\pgfmathresult{ifthenelse(\columncount==\pgfkeysvalueof{/tikz/zoomboxarray columns},1,\columncount+1)}
		        \global\let\columncount=\pgfmathresult
		        \ifblackandwhitecycle
		            \pgfmathtruncatemacro{\newpatternnumber}{\patternnumber+1}
		            \global\edef\patternnumber{\newpatternnumber}
		        \fi
		    \end{scope}
		}
		
\begin{document}
 
\title{An Energy Stable Finite-Difference Scheme for Functionalized Cahn-Hilliard Equation and its Convergence Analysis}

	\author{
Wenqiang Feng\thanks{Department of Mathematics, The University of Tennessee, Knoxville, TN 37996 (wfeng1@vols.utk.edu)}
	\and
Zhen Guan\thanks{Department of Mathematics, The University of California, Irvine, CA 92697 (zguan2@math.uci.edu)}
	\and
John Lowengrub\thanks{Department of Mathematics, The University of California, Irvine, CA  92967 (lowengrb@math.uci.edu)}
	\and		
Cheng Wang\thanks{Department of Mathematics, The University of Massachusetts, North Dartmouth, MA  02747 (Corresponding Author: cwang1@umassd.edu)}	
	\and
Steven M. Wise\thanks{Department of Mathematics, The University of Tennessee, Knoxville, TN 37996 (swise1@utk.edu)} 
}

\maketitle
	
\numberwithin{equation}{section}
	
\begin{abstract}
We present and analyze an unconditionally energy stable and convergent finite difference scheme for the Functionalized Cahn-Hilliard equation. One key difficulty associated with the energy stability is based on the fact that one nonlinear energy functional term in the expansion appears as non-convex, non-concave. To overcome this subtle difficulty, we add two auxiliary terms to make the combined term convex, which in turns yields a convex-concave decomposition of the physical energy.  As a result, an application of the convex splitting methodology assures both the unique solvability and the unconditional energy stability of the proposed numerical scheme. To deal with a 4-Laplacian solver in an $H^{-1}$ gradient flow at each time step, we apply an efficient preconditioned steepest descent algorithm to solve the corresponding nonlinear systems. In addition, a global in time $H_{\rm per}^2$ stability of the numerical scheme is established at a theoretical level, which in turn ensures the full order convergence analysis of the scheme.  A few numerical results are presented, which confirm the stability and accuracy of the proposed numerical scheme. 
\end{abstract}



\section{Introduction}
The Functionalized Cahn-Hilliard (FCH) model was first derived to describe phase separation of an amphiphilic mixture in \cite{Gompper90}. More recent work may be found in~\cite{dai2013geometric, doelman2014meander, gavish2011curvature, gavish2012variational, Promislow09, promislow2015existence}, where, in particular, the FCH equations were extended to describe membrane bilayers~\cite{dai2013geometric,doelman2014meander}, membranes and networks undergoing pearling bifurcations~\cite{promislow2015existence,doelman2014meander}, the formation of pore-like and micelle network structures~\cite{gavish2011curvature,gavish2012variational,promislow2015existence}. Consider  the standard Cahn-Hilliard (CH) energy \cite{allen79,cahn61,cahn58} given by 
\begin{equation}\label{energy-CH}
{\cal F}_0(\phi) = \int_\Omega\left\{\frac{1}{4}  \phi^4   
   	   - \frac12 \phi^2 
	   + \frac{\varepsilon^2}{2} \Bigl| \nabla \phi \Bigr|^2 \right\}d{\bf x} , 
\end{equation} 
with $\Omega\subset \mathbb{R}^D$, $D = 2$ or 3. The phase variable $\phi:\Omega\rightarrow \mathbb{R}$ is the concentration field, and $\varepsilon$ is the width of interface. We assume that $\Omega = (0,L_x)\times(0,L_y) \times (0,L_z)$ and $\phi$ is $\Omega$-periodic. The Cahn-Hillard  chemical potential becomes
	\begin{equation}
	\label{chem-poten}
\mu_0 := \delta_\phi {\cal F}_0 =  \phi^3 - \phi   - \varepsilon^2 \Delta \phi,
	\end{equation}
where $\delta_\phi$ denotes the variational derivative with respect to $\phi$.
Herein we consider a dimensionless energy of a binary mixture: 
	\begin{equation}
	\label{energy-FCH}
{\cal F} (\phi) = \frac{\varepsilon^{-2}}{2} \int_\Omega \mu_0^2 d{\bf x}  - \eta {\cal F}_0 (\phi),
	\end{equation}
where $\eta$ is a parameter.  When $\eta > 0$ and $\eta < 0$, \eqref{energy-FCH} represents the  FCH energy \cite{doelman2014meander, Hsu83,Promislow09} and the Cahn-Hilliard-Willmore (CHW) energy \cite{torabi2009new, torabi2007new, wise2007solving},  respectively. Furthermore, \eqref{energy-FCH} represents the strong FCH energy when $\eta = \varepsilon^{-1}$ and weak FCH energy when $\eta = 1$ \cite{doelman2014meander}.
By the definition of CH energy in \eqref{energy-CH} and chemical potential in \eqref{chem-poten}, we have 
	\begin{align}
\mu := \delta_\phi {\cal F} = & \ 3\varepsilon^{-2} \phi^5 - \left( 4 \varepsilon^{-2} +  \eta \right) \phi^3  +\left( \varepsilon^{-2}  + \eta \right) \phi + \varepsilon^{2} \Delta^2 \phi + \left( 2 + \eta \varepsilon^2 \right)  \Delta \phi
	\nonumber 
	\\
& + 6 \phi \left| \nabla \phi \right|^2   - 6 \nabla \cdot \left( \phi^2 \nabla \phi \right).  
	\end{align}
The conserved $H^{-1}$ gradient flow \cite{doelman2014meander,jones2013development,Promislow09} is given by 
\begin{equation}\label{equation-FCH-Hm1}
\partial_t\phi = \nabla\cdot\left(M(\phi)\nabla\mu\right),
\end{equation}
where $M(\phi) > 0$ is a diffusion mobility.

The FCH equation \eqref{equation-FCH-Hm1} is a sixth-order, highly nonlinear parabolic equation. Numerical approximation of \eqref{equation-FCH-Hm1} is very challenging because of the high derivative order and the high nonlinearity. One of the biggest challenges is to overcome the numerical stiffness encountered with time-space discretization. Roughly speaking, since the equation is sixth-order parabolic, an explicit numerical scheme is expected to encounter a severe CFL condition: $s\le C h^6$, with $s$ and $h$ the time and space step sizes. On the other hand, a fully implicit scheme, such as the backward Euler method, may still be only conditionally stable, and, very likely, will only be conditionally solvable. Ideally, one would like a scheme that preserves some of the time-invariant quantities of the PDE, such as mass conservation and the energy dissipation rate. The first invariant is easily maintained, while the second one is a major challenge. \emph{Often, one attempts only to design a scheme that will dissipate the free energy at the numerical level, without directly controlling  the rate of dissipation.} In particular, one wants $\mathcal{F}(\phi^{k+1}) \le \mathcal{F}(\phi^k)$, where $\phi^k$ is the approximated phase variable at time step $k$, given some mild CFL condition, or no CFL condition whatever.  The energy dissipativity imparts some stability notion for the PDE and the numerical method, as we will see. If $\mathcal{F}(\phi^{k+1}) \le \mathcal{F}(\phi^k)$, for all $k\ge 1$, with no condition on the time step size, we say that the scheme is \emph{unconditionally strongly energy stable}. Finally, for large-scale calculations in practice, novel efficient numerical linear and nonlinear solvers have to be carefully developed. We will address this issue in the paper as well.

There have been a few previous  works on the numerical approximation of the FCH equation.  In \cite{chen2012efficient}, Chen \emph{et al.}~presented an efficient linear, first-order (in time) spectral-Galerkin method for the FCH equation.  Their scheme, which utilized linear stabilization terms, is unconditionally solvable, but not necessary energy stable. Jones studied a semi-implicit numerical scheme for the FCH equation in his PhD~thesis \cite{jones2013development}. He proved the energy stability of his scheme but not the unique solvability. In a more recent work, \cite{Christlieb14high}, fully implicit schemes with pseudo-spectral approximation in space for the FCH equation are proposed. While the authors of~\cite{Christlieb14high} proved neither energy stability nor solvability, they did carry out several tests to show the accuracy and efficiency of their methods.  In another work \cite{guo2015local}, Guo \emph{et al.}~presented a local liscontinuous Galerkin (LDG) method to overcome the difficulty associated with the higher order spatial derivatives. Energy stability was established for the semi-discrete (time-continuous) scheme. Their fully discrete scheme was based on the time discretization in~\cite{chen2012efficient}.  In~\cite{wangx16} the authors developed a Runge-Kutta exponential time integration  (EKR) method for the diffuse Willmore flow, an equation that is closely related to the FCH and CHW models~\eqref{equation-FCH-Hm1}. This method works well when $M\equiv 1$, but may need to be significantly modified otherwise.  It enables one to generate high-order single-step methods, which have a significant advantage over multistep methods when the time step changes adaptively.  To our knowledge, there has been no rigorous convergence analysis for the FCH model in the existing literature.

In this paper we propose and analyze an efficient computational scheme for solving the FCH equation primarily, though the theory will be applicable to the CHW equation as well. We use the convex splitting method, which treats that part of the chemical potential $\mu$ coming from the convex part implicitly in the time discretization, and that coming from the concave part, explicitly. It has been a popular approach for gradient flows, since it ensures the unique solvability and unconditional stability; see the related works \cite{aristotelous13b, baskaran2013energy, chen12, chen16, chen14, diegel2015analysis, diegel15b, guo15a, hu2009stable, shen12, wang2011energy, wang10a, wise09a} for a wide class of phase field models.  For the FCH equation (\ref{equation-FCH-Hm1}), the key difficulty is that the energy does not have a straightforward convex-concave splitting. To overcome this difficulty, we add and subtract a non-trivial auxiliary term in the energy functional. Subsequently, a convex-concave decomposition for the FCH energy is available, and the first-order-in-time convex-splitting scheme is automatically available; both the unique solvability and unconditional energy stability follow immediately. 

As a result of the proposed numerical scheme, a 4-Laplacian term has to be solved in an $H^{-1}$ gradient flow at each time step in the finite difference approximation, which turns out to be very challenging. We apply a preconditioned steepest descent (PSD) solver, recently proposed and analyzed in \cite{feng2016preconditioned}, to solve the nonlinear system. The main idea is to use a linearized version of the nonlinear operator as a pre-conditioner, or in other words, as a metric for choosing the search direction. The convexity of the nonlinear energy functional assures the geometric convergence the PSD iteration sequence. In practice, only few constant-coefficient Poisson-like equations need to be solved at each iteration stage, which greatly improves the numerical efficiency over Newton-type methods.

On the theoretical side, we also present a global in time $H_{\rm per}^2$ stability of the numerical scheme. This uniform in time bound enables us to derive the full order convergence analysis, with first order temporal accuracy and second order spatial accuracy. In addition, such a convergence is unconditional, without any requirement between the time step size $s$ and the spatial mesh $h$. To the authors' knowledge, this is the first such theoretical result for the FCH/CHW model. 

This article is organized as follows. In Section~\ref{sec:1stCS} we describe the convex-splitting framework with auxiliary terms and the global in time $H_{\rm per}^2$ stability of the numerical scheme. In Section~\ref{sec:convergence} we present the main results of our analysis, including the consistency, stability and convergence of our scheme. The finite difference approximation is outlined in Section~\ref{sec:fdm}, and the preconditioned steepest descent solver is formulated in \ref{sec:psd}. Subsequently, a few numerical results are presented in Section~\ref{sec:num}, respectively. Finally, we give some concluding remarks and some future work in Section~\ref{sec:conclusion}.

	\section{The first order convex splitting scheme}
	\label{sec:1stCS} 

	\subsection{Some preliminaries}
For simplicity of presentation, we denote $( \cdot , \cdot )$ as the standard $L^2$ inner product and $\| \cdot \|$ as the standard $L^2$ norm. We use the notation $H_{\rm per}^{-1}(\Omega) = \left(H_{\rm per}^1(\Omega)\right)^*$, and $\langle  \, \cdot \, , \, \cdot \, \rangle$ is the duality paring between $H_{\rm per}^{-1}(\Omega)$ and $H_{\rm per}^1(\Omega)$.  To define an energy for this system we need a norm on a subspace of $H_{\rm per}^{-1}(\Omega)$. With $\mathring{L}^2(\Omega)$ denoting those function in $L^2(\Omega)$ with zero mean, we set
	\begin{equation}
\mathring{H}_{\rm per}^1(\Omega) = H_{\rm per}^1(\Omega)\cap \mathring{L}^2(\Omega), \quad \mathring{H}_{\rm per}^{-1}(\Omega) :=\left\{v\in H_{\rm per}^{-1}(\Omega) \ \middle| \ \langle v , 1 \rangle = 0  \right\}.
	\end{equation}
Next, we define a linear operator $\mathsf{T} : \mathring{H}_{\rm per}^{-1}(\Omega) \rightarrow \mathring{H}_{\rm per}^1(\Omega)$ via the following variational problem: given $\zeta\in \mathring{H}_{\rm per}^{-1}(\Omega)$, find $\mathsf{T}(\zeta)\in \mathring{H}_{\rm per}^1(\Omega)$ such that
	\begin{equation}
\iprd{\nabla \mathsf{T}(\zeta)}{\nabla\chi} = \langle \zeta, \chi\rangle , \qquad \forall \chi\in \mathring{H}_{\rm per}^1(\Omega).
	\label{defn-T-operator}
	\end{equation}
$\mathsf{T}$ is well-defined, as guaranteed by the Riesz Representation Theorem.  The following facts can be easily established \cite{diegel2015analysis,feng2016preconditioned}. 

\begin{lem}
\label{lem-negative-norm}
Let $\zeta,\, \xi \in\mathring{H}_{\rm per}^{-1}(\Omega)$ and, for such functions, we set
\begin{equation}
\left(\zeta,\xi\right)_{\mathring{H}_{\rm per}^{-1}} :=\iprd{\nabla \mathsf{T}(\zeta)}{\nabla\mathsf{T}(\xi)} =\langle\zeta,\mathsf{T}(\xi)\rangle  = \langle \xi, \mathsf{T}(\zeta)\rangle.
\label{crazy-inner-product}
\end{equation}
Then, $\left(\, \cdot\, ,\, \cdot\, \right)_{\mathring{H}_{\rm per}^{-1}}$ defines an inner product on $\mathring{H}_{\rm per}^{-1}(\Omega)$, and the induced norm is equivalent to (in fact, equal to) the operator norm:
\begin{equation}
\norm{\zeta}{\mathring{H}^{-1}_{\rm per}} := \sqrt{\left(\zeta,\zeta\right)_{\mathring{H}_{\rm per}^{-1}}} = \sup_{0\ne \chi\in\mathring{H}^1_{\rm per}} \frac{\langle \zeta , \chi\rangle}{\norm{\nabla\chi}{}} . 
\label{crazy-norm-minus-one}
\end{equation}
Consequently, we have $\left|\langle \zeta , \chi\rangle\right| \le \norm{\zeta}{\mathring{H}^{-1}_{\rm per}} \nrm{\nabla\chi}$, for all $\chi\in H_{\rm per}^1(\Omega)$ and $\zeta\in\mathring{H}_{\rm per}^{-1}(\Omega)$.  Furthermore, for all $\zeta\in \mathring{L}^2(\Omega)$, we have the Poincar\'e type inequality: $\norm{\zeta}{\mathring{H}^{-1}_{\rm per}} \le C  \norm{\zeta}{}$, for some $C>0$. 
\end{lem}

\subsection{The convex-concave energy decomposition with auxiliary terms} 
For any $\phi \in H_{\rm per}^2(\Omega)$, the FCH energy in \eqref{energy-FCH} may be expanded as 
	\begin{align}
{\cal F} (\phi) = & \  \frac{\varepsilon^{-2} }{2}  \nrm{ \phi }_{L^6}^6  - \left( \varepsilon^{-2} + \frac{\eta}{4} \right) \nrm{ \phi }_{L^4}^4 + \left( \frac{\varepsilon^{-2} }{2} + \frac{\eta}{2} \right) \nrm{ \phi }^2  + \frac{\varepsilon^2}{2} \nrm{ \Delta \phi }^2
	\nonumber 
	\\
& - \left( 1 + \frac{\eta \varepsilon^2}{2} \right)  \nrm{ \nabla \phi }^2 
  + 3 \int_\Omega \phi^2 \left| \nabla \phi \right|^2 d{\bf x} .
	\label{energy-FCH-1}
	\end{align} 
Unlike the energies for the AC~\cite{feng2014analysis}, CH~\cite{aristotelous13b, diegel15b, eyre98, feng2016analysis, guo15a}, Phase Field Crystal (PFC)~\cite{baskaran2013energy, hu2009stable, wang2011energy, wise09a}, epitaxial thin~\cite{chen12, chen14, shen12, wang10a} equations, the convex splitting idea cannot be directly applied to the FCH energy~\eqref{energy-FCH}. The main difficulty is associated with the last term in \eqref{energy-FCH-1},
	\begin{equation} 
\mathcal{G}(\phi) := \int_\Omega 3\phi^2 \left| \nabla \phi \right|^2 d{\bf x} ,   \label{convexity-G-1}
	\end{equation} 
which is neither convex nor concave. To overcome this difficulty, we perform a careful analysis for the following energy functional: 
	\begin{equation}
\mathcal{H}(\phi) := \int_\Omega \left( A ( \phi^4 + | \nabla \phi |^4 ) + 3 \phi^2 \left| \nabla \phi \right|^2  \right) d{\bf x} .
	\label{convexity-H-1}
	\end{equation} 

	\begin{lem}
	\label{lem:Hconv} 
$\mathcal{H}:W^{1,4}_{\rm per}(\Omega)\to\mathbb{R}$ is convex provided that $A \ge 1$. 
	\end{lem} 

	\begin{proof}  
We denote $g (\phi) := 3 \phi^2 \left| \nabla \phi \right|^2$ and $h (\phi) := A ( \phi^4 + | \nabla \phi |^4 ) + g(\phi)$, so that $\mathcal{G}(\phi) = \int_\Omega g (\phi)\, d{\bf x}$ and $\mathcal{H}(\phi) = \int_\Omega h (\phi)\, d{\bf x}$. Based on the pointwise inequalities, 
	\begin{equation*} 
\left( \frac{\phi_1 + \phi_2 }{2}  \right)^2  \le  \frac{\phi_1^2 + \phi_2^2 }{2}  ,  \quad  \left| \nabla \left( \frac{ \phi_1 + \phi_2 }{2}  \right) \right|^2  \le  \frac{|\nabla \phi_1 |^2 + | \nabla \phi_2 |^2 }{2} ,    \quad  \forall \phi_1 , \, \phi_2 , 
	\end{equation*} 
which come from the convexity of $q_2(x) =x^2$ and $r_2(\x) = \x\cdot\x$, we find that
	\begin{equation*} 
g \left( \frac{\phi_1 + \phi_2 }{2}  \right)   = 3 \left( \frac{\phi_1 + \phi_2 }{2}  \right)^2  \left| \nabla \left( \frac{ \phi_1 + \phi_2 }{2}  \right) \right|^2  \le  3 \frac{\phi_1^2 + \phi_2^2 }{2} \cdot \frac{ |\nabla \phi_1 |^2 + | \nabla \phi_2 |^2 }{2} . 
	\end{equation*}  
A careful comparison with $\frac{g (\phi_1) + g (\phi_2)}{2} 
= \frac{3 \phi_1^2 \left| \nabla \phi_1 \right|^2 + 3 \phi_2^2 \left| \nabla \phi_2 \right|^2}{2}$ shows that 
\begin{eqnarray} 
\frac{g (\phi_1) + g (\phi_2)}{2} - g \left( \frac{\phi_1 + \phi_2 }{2}  \right)   
&\ge&  \frac{ 3 ( \phi_1^2 - \phi_2^2 ) ( |\nabla \phi_1 |^2 - | \nabla \phi_2 |^2 ) }{4}   \nonumber 
\\
&\ge& - \frac38 \left(  ( \phi_1^2 - \phi_2^2 )^2 
  +  ( |\nabla \phi_1 |^2 - | \nabla \phi_2 |^2 )^2 \right) .  
\label{lemma 1-3}
\end{eqnarray} 
Meanwhile, the convexity of $q_4 (x) =x^4$ and $r_4 (\x) = | \x|^4$ indicates the following inequalities: 
\begin{eqnarray} 
\frac{\phi_1^4 + \phi_2^4 }{2} - \left( \frac{\phi_1 + \phi_2 }{2}  \right)^4  
&\ge& \frac{3}{8} ( \phi_1^4 + \phi_2^4 - 2 \phi_1^2 \phi_2^2 ) 
= \frac{3}{8} ( \phi_1^2  - \phi_2^2 )^2 ,   
\label{lemma 1-4-1}
\end{eqnarray} 
and
\begin{eqnarray}
\frac{ | \nabla \phi_1 |^4 + | \nabla \phi_2 |^4 }{2} 
- \left| \nabla \left( \frac{\phi_1 + \phi_2 }{2}  \right) \right|^4  
&\ge& \frac{3}{8} ( | \nabla \phi_1 |^4 + | \nabla \phi_2 |^4 
 - 2 | \nabla \phi_1 |^2 \cdot | \nabla \phi_2 |^2 )   \nonumber 
\\
&=& \frac{3}{8} ( | \nabla \phi_1 |^2  - | \nabla \phi_2 |^2 )^2 .   
\label{lemma 1-4-2}
\end{eqnarray} 
A combination of \eqref{lemma 1-3}, \eqref{lemma 1-4-1} and \eqref{lemma 1-4-2} implies that 
	\begin{equation*} 
\frac{h (\phi_1) + h (\phi_2)}{2} - h \left( \frac{\phi_1 + \phi_2 }{2}  \right)  \ge 0 ,  \quad  \forall \phi_1 , \, \phi_2 ,  \quad 
	\end{equation*} 
provided that $A \ge 1$. As a result, an integration over $\Omega$ leads to the following fact: 
	\begin{equation*} 
\frac{\mathcal{H} (\phi_1) + \mathcal{H} (\phi_2)}{2} - \mathcal{H} \left( \frac{\phi_1 + \phi_2 }{2}  \right)  \ge 0 ,  \quad\forall \phi_1 , \, \phi_2 ,  \quad  \mbox{if $A \ge 1$} . 
	\end{equation*} 
The convexity of $H$ is assured under the condition $A \ge 1$.
	\end{proof}  
 
	\begin{cor}
The energy ${\cal F}:H^2_{\rm per}(\Omega)\to\mathbb{R}$ possesses a convex splitting over $H^2_{\rm per}(\Omega)$. In particular, 
	\begin{equation} 
{\cal F} (\phi) = {\cal F}_c (\phi) - {\cal F}_e (\phi) ,
	\label{eqn:convexity-eng} 
	\end{equation} 
with
	\begin{equation} 
{\cal F}_c (\phi) := \int_\Omega  \left\{ \frac{\varepsilon^{-2}}{2}  \phi^6 + \left( \frac{\varepsilon^{-2} }{2} + \frac{\eta}{2} \right) \phi^2 + \frac{\varepsilon^2}{2} (\Delta \phi)^2\right. + \left. A ( \phi^4 + | \nabla \phi |^4 ) + 3\phi^2 \left| \nabla \phi \right|^2  \right \} d{\bf x}  ,
	\label{convexity-3}   
	\end{equation}
and
	\begin{equation}
{\cal F}_e (\phi) := \int_\Omega  \left\{ \left( \epsilon^{-2} + \frac{\eta}{4} \right) \phi^4  + \left(  1 + \frac{\eta \varepsilon^2}{2} \right)  \left| \nabla \phi \right|^2  + A ( \phi^4 + | \nabla \phi |^4 ) \right\} d{\bf x} ,
	\label{convexity-4}  
	\end{equation} 
where both ${\cal F}_c, {\cal F}_e:H^2_{\rm per}(\Omega)\to\mathbb{R}$ are strictly convex provided $A\ge 1$.
	\end{cor}
We recall the following proposition from \cite{wise09a}:
	\begin{prop}
	\label{splitting-ineq}
Suppose that $\phi,\, \psi \in H_{\rm per}^4(\Omega)$ and that ${\cal F}$ admits a (not necessarily unique) convex splitting into ${\cal F} = {\cal F}_c-{\cal F}_e$ then
	\begin{equation}
	\label{eqn:splitting-ineq}
{\cal F}(\phi)-{\cal F}(\psi) \le \left(\delta_\phi {\cal F}_c(\phi)-\delta_\phi {\cal F}_e(\psi), \phi-\psi \right).
	\end{equation}
If $\phi,\, \psi \in H_{\rm per}^2(\Omega)$ only, then \eqref{eqn:splitting-ineq} can be interpreted in the weak sense.
	\end{prop}
	
\subsection{The first order convex splitting scheme}

Based on the convex-concave decomposition in \eqref{convexity-3} and \eqref{convexity-4} for the physical energy ${\cal F} (\phi)$, we consider the following semi-implicit, first-order-in-time, convex splitting  scheme:
	\begin{equation}
	\label{scheme-CS-Hm1}
\phi^{k+1}-\phi^k = s\nabla\cdot\left(M(\phi^k)\nabla\tilde\mu\right) ,\quad \tilde\mu\left(\phi^{k+1},\phi^k\right) := \delta_\phi {\cal F}_c(\phi^{k+1}) - \delta_\phi {\cal F}_e(\phi^k),
	\end{equation}
where, precisely,
	\begin{align}
\tilde\mu\left(\phi^{k+1},\phi^k\right) = & \   3 \varepsilon^{-2} ( \phi^{k+1} )^5  + 4 A ( \phi^{k+1} )^3  + ( \varepsilon^{-2}  + \eta ) \phi^{k+1}  + \varepsilon^2 \Delta^2 \phi^{k+1} 
	\nonumber 
	\\
&  + 6 \phi^{k+1} \left| \nabla \phi^{k+1} \right|^2 - 6 \nabla \cdot \left( ( \phi^{k+1} )^2 \nabla \phi^{k+1} \right)- 4 A \nabla \cdot \left( | \nabla \phi^{k+1} |^2 \nabla \phi^{k+1} \right) 
	\label{CS-Hm1e}   
	\\
&  - ( 4 \varepsilon^{-2} +  \eta ) ( \phi^k )^3 + ( 2 + \eta \varepsilon^2 )  \Delta \phi^k  - 4 A ( \phi^k )^3  + 4 A \nabla \cdot \left( | \nabla \phi^k |^2 
  \nabla \phi^k \right). 
  	\nonumber 
	\end{align}
The scheme may be expressed in a weak form as follows: find the pair $(\phi,\mu)\in H^2_{\rm per}(\Omega)\times H^1_{\rm per}(\Omega)$ such that
	\begin{align}
(\phi,\nu) + s(M \nabla\mu,\nabla\nu) & \, = (g,\nu),
	\\
 \left( 3 \varepsilon^{-2} \phi^5  + 4 A \phi^3  + ( \varepsilon^{-2}  + \eta ) \phi, \psi\right)  + \varepsilon^2 ( \Delta \phi,\Delta\psi) + 6 (\phi \left| \nabla \phi \right|^2,\psi) &
 	\\
+ 6 \left( \phi^2 \nabla \phi,\nabla\psi \right)+ 4 A  \left( | \nabla \phi |^2 \nabla \phi,\nabla\psi \right)  -(\mu,\psi) & \, = (f,\psi),
	\end{align}
where $g = \phi^k$, $M = M(\phi^k)$, and
	\[
f =  \delta_\phi\mathcal{F}_e(\phi^k) = ( 4 \varepsilon^{-2} +  \eta ) ( \phi^k )^3 - ( 2 + \eta \varepsilon^2 )  \Delta \phi^k  + 4 A ( \phi^k )^3  - 4 A \nabla \cdot \left( | \nabla \phi^k |^2   \nabla \phi^k \right).
	\]
Observe that, if $\phi^k\in H^2_{\rm per}(\Omega)$ is given, we have $g,f\in L^2_{\rm per}(\Omega) = L^2(\Omega)$.

	\begin{thm}
	\label{thm:eng-decay} 
The convex splitting scheme \eqref{scheme-CS-Hm1} is uniquely solvable and unconditionally energy stable: ${\cal F} (\phi^{k+1}) \le {\cal F} (\phi^k)$. In particular, if $\phi^k\in H^2_{\rm per}(\Omega)$, then $\phi^{k+1}\in H^2_{\rm per}(\Omega)$.
	\end{thm}
	\begin{proof}
The existence and unique solvability follows from standard convexity analyses.  For the stability, let $\phi=\phi^{k+1}$ and $\psi=\phi^k$ in \eqref{eqn:splitting-ineq} to find
	\begin{eqnarray*}
{\cal F}(\phi^{k+1})-{\cal F}(\phi^k) &\le& \left(\delta_\phi {\cal F}_c(\phi^{k+1})-\delta_\phi {\cal F}_e(\phi^k), \phi^{k+1}-\phi^k \right)
	\nonumber
	\\
&=& s\left(\tilde\mu, \nabla\cdot\left(M(\phi^k)\nabla\tilde\mu\right) \right) 
= -s\left(\nabla\tilde\mu , M(\phi^k)\nabla\tilde\mu\right) \le 0,
	\end{eqnarray*} 
where we have interpreted the right-hand-side of \eqref{eqn:splitting-ineq} in the weak sense. 
	\end{proof}

	\subsection{Global-in-time $H_{\rm per}^2$ stability of the numerical scheme} 
	
For simplicity, we will take the mobility $M \equiv 1$ in the remainder of the  paper. 
	\begin{lem}
There are constants $C_0,C_1>0$ such that, for all $\phi\in H^2_{\rm per}(\Omega)$,
	\begin{equation}
\frac{\varepsilon^{-2} }{6}  \nrm{ \phi }_{L^6}^6 + C_0 \varepsilon^2 \nrm{ \phi }_{H_{\rm per}^2}^2 \le  {\cal F} (\phi)+C_1 .
	\label{energy-FCH-2}
	\end{equation}
	\end{lem}
	\begin{proof}
For the concave diffusion term in \eqref{energy-FCH-1}, an application of Cauchy's inequality shows that 
	\begin{equation} 
\nrm{ \nabla \phi }^2 = \int_\Omega \phi \cdot \Delta \phi\,  d{\bf x} 
\le \nrm{ \phi } \cdot \nrm{ \Delta \phi } 
\le   \frac{\varepsilon^2}{4 ( 1 + \frac{\eta \varepsilon^2}{2} ) } \nrm{ \Delta \phi }^2 
+ \frac{1 + \frac{\eta \varepsilon^2}{2}}{\varepsilon^2} \nrm{ \phi }^2 ,  
\quad \forall \, \eta  > 0 .  \label{H2-stab-1}
	\end{equation}
Then we obtain  
	\begin{equation} 
( 1 + \frac{\eta \varepsilon^2}{2} ) \nrm{ \nabla \phi }^2  \le  \frac{\varepsilon^2}{4} \nrm{ \Delta \phi }^2 + C_2 \nrm{ \phi }^2 ,
	\label{H2-stab-2}
	\end{equation}
with $C_2 :=  ( 1 + \frac{\eta \varepsilon^2}{2} )^2 \varepsilon^{-2} =O(\varepsilon^{-2})$.  Applications of H\"older's inequality imply that 
	\begin{equation*}
\nrm{ \phi }_{L^6}  \ge  \frac{1}{ | \Omega |^{1/12} } \nrm{ \phi }_{L^4},\quad
\nrm{ \phi }_{L^6}  \ge  \frac{1}{ | \Omega |^{1/3} } \nrm{ \phi } .
	\end{equation*}
Now, define $ C_3 := C_2 - \left( \frac{\varepsilon^{-2} }{2} + \frac{\eta}{2} \right) 
+ 1 >0$; we note that $C_3=O (\varepsilon^{-2} )$. As a consequence of the last two inequalities, we get 
	\begin{eqnarray}   
\frac16 \nrm{ \phi }_{L^6}^6  &\ge&  \frac{1}{ 6 | \Omega |^{1/2} } \nrm{ \phi }_{L^4}^6 \ge ( 1 + \frac{\eta \varepsilon^2}{4} ) \nrm{ \phi }_{L^4}^4 - C_4  ,
	\label{H2-stab-4-1}
	\\
\frac16 \nrm{ \phi }_{L^6}^6  &\ge&  \frac{1}{ 6 | \Omega |^{2} } \nrm{ \phi } ^6  \ge  \varepsilon^2 C_3 \nrm{ \phi } ^2 - C_5  ,
	\label{H2-stab-4-2}
	\end{eqnarray}
for some constants $C_4,C_5>0$, which are of order 1, where Young's inequality was repeated applied. Therefore, a combination of \eqref{energy-FCH-1}, \eqref{H2-stab-2}, \eqref{H2-stab-4-1} and \eqref{H2-stab-4-2} yields 
	\begin{eqnarray}
{\cal F} (\phi) &\ge& \frac{\varepsilon^{-2} }{6}  \nrm{ \phi }_{L^6}^6 + \nrm{ \phi }^2 + \frac{\varepsilon^2}{4} \nrm{ \Delta \phi }^2 - C_1, 
	\nonumber 
	\\
&\ge&  \frac{\varepsilon^{-2} }{6}  \nrm{ \phi }_{L^6}^6 + C_0 \varepsilon^2 \nrm{ \phi }_{H_{\rm per}^2}^2 - C_1 ,
	\end{eqnarray}  
where $C_1:= \varepsilon^{-2} \left( C_4 + C_5 \right) =O(\varepsilon^{-2}$) and the elliptic regularity estimate $\nrm{ \phi }_{H^2}^2 \le C_0 ( \nrm{ \phi }^2  + \nrm{ \Delta \phi }^2 )$  was applied in the second step.
	\end{proof}

	\begin{cor}
Suppose that $\phi_0\in H^2_{\rm per}(\Omega)$. For any positive integer $k$, we have 
	\begin{eqnarray} 
\nrm{ \phi^k }_{H_{\rm per}^2} \le C_6 := \frac{{\cal F} (\phi^0) + C_1 }{C_0\varepsilon^2} .
  \label{H2-stab-5}
\end{eqnarray}
	\end{cor}
	\begin{proof}
The unconditional energy stability in Theorem~\ref{thm:eng-decay} implies that, for any positive integer $k$,
	\begin{equation} 
{\cal F} (\phi^k) \le {\cal F} (\phi^0). 
	\label{energy-stab-1}
	\end{equation}
A combination of \eqref{energy-FCH-2} and \eqref{energy-stab-1} yields the result. 
	\end{proof}

	\begin{rem}
Note that the constant $C_6$ is independent of $k$ and $s$, but does depends on $\varepsilon$. In particular, $C_6=O (\varepsilon^{-4})$.
	\end{rem}

\section{Convergence analysis}\label{sec:convergence} 
\subsection{Main result} 
The convergence result is stated in the following theorem. 
The following regularity classes are introduced: 

\begin{eqnarray} 
  &&
  \mathcal{R}_1 = C^2 ( [0,T]; C^0_{\rm per} (\Omega)) \cap C^1( [0,T]; C^4_{\rm per} (\Omega)) 
  \cap L^\infty ( [0,T]; C^6_{\rm per} (\Omega) ) , 
  \label{regularity assumption-1} 
\\
  &&
  \mathcal{R}_2 = C^2 ( [0,T]; C^0_{\rm per} (\Omega)) \cap C^1( [0,T]; C^4_{\rm per} (\Omega)) 
  \cap L^\infty ( [0,T]; C^8_{\rm per} (\Omega) ) . 
  \label{regularity assumption-2}    
\end{eqnarray}    


	\begin{thm} \label{convergence}
Let $\Phi \in \mathcal{R}_1$ be the exact periodic solution of the FCH equation \eqref{equation-FCH-Hm1} with the initial data $\Phi(0) =\phi_0\in H^2_{\rm per}(\Omega)$. Suppose $\phi$ is the space-continuous  numerical solution of \eqref{scheme-CS-Hm1}.   Then the following error estimate is valid:
	\begin{eqnarray} 
\| \Phi - \phi \|_{\ell^\infty(0,T; \mathring{H}_{\rm per}^{-1})} +  \| \Phi - \phi \|_{\ell^2 (0,T; H_{\rm per}^2)} \le C s ,  
\label{convergence-1}
	\end{eqnarray}     
where the constant $C>0$ depends only on the regularity of the exact solution. 
	\end{thm} 

\subsection{Proof of the main result}
\subsubsection{Consistency analysis} 

Define $\Phi^k = \Phi(\, \cdot \, ,t_k)$. A detailed Taylor expansion implies the following truncation error: 
\begin{eqnarray} 
\frac{\Phi^{k+1} - \Phi^k}{s} &=&  \Delta \biggl( 3 \varepsilon^{-2} ( \Phi^{k+1} )^5 
  - ( 4 \varepsilon^{-2} +  \eta ) ( \Phi^k )^3  
  + ( \varepsilon^{-2}  + \eta ) \Phi^{k+1}   
  + \varepsilon^2 \Delta^2 \Phi^{k+1} 
  \nonumber 
\\
&& + ( 2 + \eta \varepsilon^2 )  \Delta \Phi^k  + 6 \Phi^{k+1} \left| \nabla \Phi^{k+1} \right|^2  
   - 6 \nabla \cdot \left( ( \Phi^{k+1} )^2 \nabla \Phi^{k+1} \right)  \nonumber
\\
&& + 4 A ( \Phi^{k+1} )^3  - 4 A \nabla \cdot \left( | \nabla \Phi^{k+1} |^2 \nabla \Phi^{k+1} \right)   
\label{consistency-1} 
\\
&& - 4 A ( \Phi^k )^3  + 4 A \nabla \cdot \left( | \nabla \Phi^k |^2 \nabla \Phi^k \right)   \biggr)  + \tau^k ,  \nonumber
\end{eqnarray}
with $\nrm{ \tau^k } \le C s$ . Consequently, with an introduction of the error function
\begin{eqnarray} 
  e^k = \Phi^k - \phi^k ,   \quad \forall \, k \ge 0 , 
  \label{error function-1} 
\end{eqnarray}
we get the following evolutionary equation, by subtracting \eqref{CS-Hm1e} from \eqref{consistency-1}:  
\begin{eqnarray} 
 \frac{e^{k+1} - e^k}{s}
    &=&  \Delta \biggl( 
  3 \varepsilon^{-2} \left( ( \Phi^{k+1} )^4  + ( \Phi^{k+1} )^3 \phi^{k+1} 
  + ( \Phi^{k+1} )^2 (\phi^{k+1})^2 +  \Phi^{k+1}  (\phi^{k+1} )^3 
   + ( \phi^{k+1} )^4  \right) e^{k+1}  \nonumber 
\\
  &&   
  - ( 4 \varepsilon^{-2} +  \eta + 4 A )  \left( ( \Phi^k )^2  
  + \Phi^k  \phi^k + ( \phi^k )^2  \right) e^k   
  + ( \varepsilon^{-2}  + \eta )  e^{k+1}    
  + \varepsilon^2 \Delta^2 e^{k+1}  \nonumber 
\\
  &&  
  + ( 2 + \eta \varepsilon^2 )  \Delta e^k
  + 6 e^{k+1} \left| \nabla \Phi^{k+1} \right|^2  
  + 6 \phi^{k+1} \left( \nabla ( \Phi^{k+1} + \phi^{k+1} )  
   \cdot \nabla e^{k+1}  \right)    \nonumber 
\\
  &&   
  - 6 \nabla \cdot \left( 
  ( \Phi^{k+1} + \phi^{k+1} ) e^{k+1} \nabla \Phi^{k+1}
  + ( \phi^{k+1} )^2 \nabla e^{k+1} \right)  \nonumber 
\\
  &&  
  + 4 A ( \left( ( \Phi^{k+1} )^2  
  + \Phi^{k+1}  \phi^{k+1} + ( \phi^{k+1} )^2  \right) 
  e^{k+1}   \nonumber 
\\
  &&   
- 4 A \nabla \cdot \left(  ( \nabla ( \Phi^{k+1} + \phi^{k+1} )  
   \cdot \nabla e^{k+1} )  \nabla \Phi^{k+1} 
  + | \nabla \phi^{k+1} |^2 \nabla e^{k+1}  \right)  \nonumber 
\\
  &&   
   + 4 A \nabla \cdot \left( ( \nabla ( \Phi^k + \phi^k )  
   \cdot \nabla e^k )  \nabla \Phi^k 
  + | \nabla \phi^k |^2 \nabla e^k \right)   \biggr)  
   + \tau^k  .   \label{consistency-2} 
\end{eqnarray}

In addition, from the PDE analysis for the FCH equation and the global in time $H_{\rm per}^2$ stability \eqref{H2-stab-5} for the numerical solution,  we also get the $L^\infty$, $W^{1,6}$ and $H_{\rm per}^2$ bounds for both the exact solution and numerical solution, uniform in time: 
	\begin{equation} 
\| \Phi^k \|_{L^\infty } ,  \  \| \Phi^k \|_{W^{1,6} }  , \  \| \Phi^k \|_{H_{\rm per}^2 }  \le C_7 ,  \quad  \| \phi^k \|_{L^\infty } ,  \  \| \phi^k \|_{W^{1,6} }  , \  \| \phi^k \|_{H_{\rm per}^2 }  \le C_7 ,  \quad  \forall \, k \ge 0 ,
	\label{consistency-bound-1}
	\end{equation}
where the 3-D embeddings of $H_{\rm per}^2$ into $L^\infty$ and into $W^{1, 6}$ have been applied. Also note that $C_7$ and $C_8$ are time independent constants, that depend on $\varepsilon$ as $O (\varepsilon^{-4})$. 

\subsubsection{Stability and convergence analysis} 

First, we recall that the exact solution to the FCH equation \eqref{equation-FCH-Hm1} is mass conservative: 
\begin{equation} 
  \int_\Omega \Phi ({\bf x} , t)  \, d {\bf x} \equiv \int_\Omega \Phi ({\bf x}  , 0)  \, d {\bf x} ,  \quad \forall t > 0 .   \nonumber
  \label{mass-conserv-1}
\end{equation} 
On the other hand, the numerical solution \eqref{scheme-CS-Hm1} is also mass conservative. In turn, we conclude that the numerical error function $e^k\in  \mathring{H}_{\rm per}^{2}(\Omega)$: 
\begin{equation} 
    \overline{e^k} := \int_\Omega e^k  \, d  {\bf x}
    = \int_\Omega e^0 = 0 , \quad \mbox{since $e^0 \equiv 0$} .\nonumber  \label{error-mean-1} 
\end{equation} 
Consequently, we define $\psi^k := (- \Delta)^{-1} e^k\in \mathring{H}_{\rm per}^{-1}(\Omega)$ as 
\begin{equation} 
  - \Delta  \psi^k  =  e^k ,  \quad \mbox{with} 
   \, \, \, \int_\Omega \psi^k  \, d {\bf x} = 0 .  \nonumber \label{varphi-1} 
	\end{equation} 
Define $I_i, i=1,\cdots, 10$ by
\begin{eqnarray}
I_1: &=& - 6 \varepsilon^{-2}   s \int_\Omega  \Bigl( ( \Phi^{k+1} )^4  + ( \Phi^{k+1} )^3 \phi^{k+1} 
  + ( \Phi^{k+1} )^2 (\phi^{k+1})^2 \nonumber\\ 
&&+  \Phi^{k+1}  (\phi^{k+1} )^3
+ ( \phi^{k+1} )^4  \Bigr) \left| e^{k+1}  \right|^2  d {\bf x},\nonumber
\\
I_2: &=& - 8 A s \int_\Omega \left( ( \Phi^{k+1} )^2  
+ \Phi^{k+1}  \phi^{k+1} + ( \phi^{k+1} )^2  \right) 
\left| e^{k+1}  \right|^2  d {\bf x},\nonumber
\\
I_3: &=& 2 ( 2 + \eta \varepsilon^2 )  s ( \nabla e^k , \nabla e^{k+1}  ),\nonumber
\\
I_4: &=& 2 ( 4 \varepsilon^{-2} +  \eta + 4 A )  s \int_\Omega 
\left( ( \Phi^k )^2  + \Phi^k  \phi^k + ( \phi^k )^2  \right) 
e^k  e^{k+1}   d {\bf x},\nonumber
\\
I_5: &=& - 12 s  \int_\Omega  | \nabla \Phi^{k+1} |^2  ( e^{k+1} )^2  d {\bf x} ,\nonumber
\\
I_6: &=& - 12 s  \int_\Omega \phi^{k+1} \left( \nabla ( \Phi^{k+1} + \phi^{k+1} )  
\cdot \nabla e^{k+1}  \right)   
e^{k+1}  d {\bf x},\nonumber
\\
I_7: &=& - 12 s  \left(( \Phi^{k+1} + \phi^{k+1} ) e^{k+1} \nabla \Phi^{k+1}
+ ( \phi^{k+1} )^2 \nabla e^{k+1} ,  
\nabla e^{k+1}  \right),\nonumber
\\
I_8: &=& - 8 A s  \left(  ( \nabla ( \Phi^{k+1} + \phi^{k+1} )\cdot \nabla e^{k+1} )  \nabla \Phi^{k+1} 
+ | \nabla \phi^{k+1} |^2 \nabla e^{k+1},
\nabla e^{k+1} \right),\nonumber
\\
I_9: &=& 8 A s  \left( ( \nabla ( \Phi^k + \phi^k )\cdot \nabla e^k )  \nabla \Phi^k 
+ | \nabla \phi^k |^2 \nabla e^k , 
\nabla e^{k+1}   \right),\nonumber
\\
I_{10}: &=& - 2 s ( \tau^k , e^{k+1} ).\nonumber
\end{eqnarray}
Therefore, taking an $L^2$ inner product with the numerical error equation \eqref{consistency-2} by $2 \psi^k$ gives 
\begin{eqnarray}  
\| e^{k+1} \|_{\mathring{H}_{\rm per}^{-1} }^2 
- \| e^k \|_{\mathring{H}_{\rm per}^{-1} }^2 
+ \| e^{k+1} - e^k \|_{\mathring{H}_{\rm per}^{-1} }^2   
+  2 ( \varepsilon^{-2}  + \eta )  s \| e^{k+1} \|^2 
+ 2 \varepsilon^2 s \| \Delta e^{k+1}  \|^2
= \sum_{i=1}^{10}I_i
,   \label{convergence-2} 
\end{eqnarray}
where integration-by-parts has been repeatedly applied. 

The local truncation error term $I_{10}$ can be bounded by the Cauchy inequality: 
\begin{eqnarray} 
    - 2 ( \tau^k , e^{k+1} )  
  \le  2  \| \tau^k \| \cdot  \| e^{k+1} \| 
  \le   \| \tau^k \|^2 + \| e^{k+1} \|^2 .
  \label{convergence-3-1} 
\end{eqnarray} 
Meanwhile, an application of weighted Sobolev inequality shows that 
	\begin{equation}
\| e^{k+1} \|   \le C_8 \| e^{k+1} \|_{\mathring{H}_{\rm per}^{-1} }^{2/3} \cdot  \| e^{k+1} \|_{\mathring{H}_{\rm per}^2 }^{1/3} \le C_9 \| e^{k+1} \|_{\mathring{H}_{\rm per}^{-1} }^{2/3} \cdot  \| \Delta e^{k+1} \|^{1/3}   ,
	\label{convergence-3-2}
	\end{equation}
where a standard estimate of elliptic regularity was applied at the second step, considering the fact that $\overline{ e^{k+1}} =0$. Subsequently, an application of Young's inequality gives 
\begin{eqnarray} 
   \| e^{k+1} \|^2   
  \le C_{10} \varepsilon^{-1} \| e^{k+1} \|_{\mathring{H}_{\rm per}^{-1} }^2  
  + \frac{\varepsilon^2}{8} \| \Delta e^{k+1} \|^2  , \nonumber   
  \label{convergence-3-3} 
\end{eqnarray} 
and its combination with (\ref{convergence-3-1}) yields 
\begin{eqnarray} 
    - 2 ( \tau^k , e^{k+1} )  
  \le   \| \tau^k \|^2 
   + C_{10} \varepsilon^{-1} \| e^{k+1} \|_{\mathring{H}_{\rm per}^{-1} }^2  
  + \frac{\varepsilon^2}{8} \| \Delta e^{k+1} \|^2  . 
  \label{convergence-3-4} 
\end{eqnarray} 

The first integral term $I_{1}$ turns out to be non-positive,
	\begin{equation} 
  I_{1} \le 0 ,  \label{convergence-4-1} 
	\end{equation} 
due to the fact that 
\begin{eqnarray} 
   ( \Phi^{k+1} )^4  + ( \Phi^{k+1} )^3 \phi^{k+1} 
  + ( \Phi^{k+1} )^2 (\phi^{k+1})^2 +  \Phi^{k+1}  (\phi^{k+1} )^3   
  + ( \phi^{k+1} )^4  \ge 0 . \nonumber \label{convergence-4-2} 
\end{eqnarray} 

Since $( \Phi^{k+1} )^2  + \Phi^{k+1}  \phi^{k+1} + ( \phi^{k+1} )^2 \ge 0$, similar estimates can be derived for $I_{2}$ and $I_{5}$: 
\begin{eqnarray} 
  I_{2} &=& - 8 A s \int_\Omega  
   \left( ( \Phi^{k+1} )^2  
  + \Phi^{k+1}  \phi^{k+1} + ( \phi^{k+1} )^2  \right) 
   \left| e^{k+1}  \right|^2  d {\bf x} \le 0 ,   \label{convergence-5-1} 
\\
  I_{5} &=& 
  - 12 s  \int_\Omega  | \nabla \Phi^{k+1} |^2   
( e^{k+1} )^2  d {\bf x}  \le 0 .  \label{convergence-5-2} 
\end{eqnarray}

  For the term $I_{3}$, we denote $C_{11} = 2 + \eta \varepsilon^2$ and observe that 
\begin{eqnarray} 
 I_{3} = 2  C_{11}  s ( \nabla e^k , \nabla e^{k+1}  )  
  \le  C_{11}  s  (  \| \nabla e^k \|^2  +  \| \nabla e^{k+1} \|^2  ) . 
  \label{convergence-6-1} 
\end{eqnarray} 
Meanwhile, a similar estimate as \eqref{convergence-3-2} could be carried out to bound $\| \nabla e^{k+1} \|$:  
\begin{eqnarray} 
   \| \nabla e^{k+1} \|  
  \le C_{12} \| e^{k+1} \|_{\mathring{H}_{\rm per}^{-1} }^{1/3} \cdot 
  \| e^{k+1} \|_{H_{\rm per}^2 }^{2/3}
  \le C_{13} \| e^{k+1} \|_{\mathring{H}_{\rm per}^{-1} }^{1/3} \cdot 
  \| \Delta e^{k+1} \|^{2/3}   ,  
  \label{convergence-6-2} 
\end{eqnarray} 
so that an application of Young's inequality leads to 
\begin{eqnarray} 
   \| \nabla e^{k+1} \|^2   
  \le C_{14} \varepsilon^{-4} \| e^{k+1} \|_{\mathring{H}_{\rm per}^{-1} }^2  
  + \frac{\varepsilon^2}{8 C_{11} } \| \Delta e^{k+1} \|^2  .    
  \label{convergence-6-3} 
\end{eqnarray} 
The term $\| \nabla e^k \|$ can be bounded in the same fashion:  
\begin{equation} 
   \| \nabla e^k \|^2   
  \le C_{15} \varepsilon^{-4} \| e^k \|_{\mathring{H}_{\rm per}^{-1} }^2  
  + \frac{\varepsilon^2}{8 C_{11}} \| \Delta e^k \|^2  .    
  \label{convergence-6-4} 
\end{equation} 
Substituting \eqref{convergence-6-3} and \eqref{convergence-6-4} into  \eqref{convergence-6-1}, we get 
\begin{equation} 
 I_{3} 
  \le  C_{16} s 
  ( \| e^{k+1} \|_{\mathring{H}_{\rm per}^{-1} }^2  
   + \| e^k \|_{\mathring{H}_{\rm per}^{-1} }^2  ) 
  + \frac{\varepsilon^2}{8} s (  \| \Delta e^{k+1} \|^2  +  \| \Delta e^k \|^2  ) . 
  \label{convergence-6-5} 
\end{equation} 

For the term $I_{4}$, we denote $C_{17} = 
4 \varepsilon^{-2} +  \eta + 4 A$. By the $L^\infty$ 
bound in \eqref{consistency-bound-1} for both the exact and numerical solutions, we see that 
\begin{equation} 
   \| ( \Phi^k )^2  + \Phi^k  \phi^k + ( \phi^k )^2 \|_{L^\infty }  
  \le  3 C_7^2 .  \label{convergence-7-1}
\end{equation}
This in turn implies that 
\begin{eqnarray} 
   I_{4} &\le&    2  C_{17}  s   
   \| ( \Phi^k )^2  + \Phi^k  \phi^k + ( \phi^k )^2 \|_{L^\infty }  
   \cdot  \| e^k \|  \cdot  \| e^{k+1}  \|   
  \nonumber  
\\
  &\le& 
   6  C_{17}  C_7^2 s  \| e^k \|  \cdot  \| e^{k+1}  \|  
  \le  3  C_{17}  C_7^2 s   (  \| e^k \|^2 +  \| e^{k+1}  \|^2 ) .  
\label{convergence-7-2} 
\end{eqnarray} 
Meanwhile, the estimate \eqref{convergence-3-3} can be performed with  alternate coefficients, so that the following inequalities are available: 
\begin{eqnarray} 
   \| e^j \|^2   
  &\le&  C_{18} \| e^j \|_{\mathring{H}_{\rm per}^{-1} }^2  
  + \frac{\varepsilon^2}{24  C_{17}  C_7^2 }  
  \| \Delta e^j \|^2  ,   \quad \mbox{for $j=k, k+1$} . 
  \label{convergence-7-3} 
\end{eqnarray} 
Subsequently, its combination with \eqref{convergence-7-2} yields 
\begin{eqnarray} 
   I_{4}   \le  C_{19}  s   
   (  \| e^k \|_{\mathring{H}_{\rm per}^{-1} }^2  
  +  \| e^{k+1} \|_{\mathring{H}_{\rm per}^{-1}  }^2  )   
  + \frac{\varepsilon^2}{8} s (  \| \Delta e^{k+1} \|^2  
  +  \| \Delta e^k \|^2  )   . 
\label{convergence-7-5} 
\end{eqnarray}

For the term $I_{6}$, we start from an application of H\"older inequality: 
\begin{eqnarray} 
  I_{6} &=& - 12 s  \int_\Omega 
  \phi^{k+1} \left( \nabla ( \Phi^{k+1} + \phi^{k+1} )  
   \cdot \nabla e^{k+1}  \right)   e^{k+1}  d {\bf x}   \nonumber 
\\
  &\le& 
  C_{20} s \| \phi^{k+1} \|_{L^\infty}  
  \cdot (  \| \nabla \Phi^{k+1} \|_{L^6}  
  + \| \nabla \phi^{k+1} \|_{L^6}  )   
   \cdot \| \nabla e^{k+1} \|_{L^{3/2} }    
  \cdot \| e^{k+1} \|_{L^6}  \nonumber 
\\
  &\le& 
  C_{21} C_7^2 s  
   \cdot \| \nabla e^{k+1} \|_{L^{3/2} }    
  \cdot \| e^{k+1} \|_{L^6} , 
  \label{convergence-8-1} 
\end{eqnarray} 
in which the $L^\infty$ and $W^{1,6}$ stability bounds  for the exact and numerical solutions were recalled in the second step of \eqref{consistency-bound-1}. Moreover, the first term 
$\| \nabla e^{k+1} \|_{L^{3/2} }$ can be bounded in the following way: 
\begin{eqnarray} 
    \| \nabla e^{k+1} \|_{L^{3/2} }  
  \le  C_{22} \| \nabla e^{k+1} \| 
  \le C_{23} \| e^{k+1} \|_{\mathring{H}_{\rm per}^{-1} }^{1/3} \cdot 
  \| \Delta e^{k+1} \|^{2/3}   ,  
  \label{convergence-8-2} 
\end{eqnarray} 
with an earlier estimate \eqref{convergence-6-2} recalled. For the second term $\| e^{k+1} \|_{L^6}$, a 3-D Sobolev embedding could be applied so that 
\begin{eqnarray} 
    \| e^{k+1} \|_{L^6} 
  \le  C_{24} \| \nabla e^{k+1} \| 
  \le C_{25} \| e^{k+1} \|_{\mathring{H}_{\rm per}^{-1} }^{1/3} \cdot 
  \| \Delta e^{k+1} \|^{2/3}   .  
  \label{convergence-8-3} 
\end{eqnarray} 
We also note that the zero-mean property for $e^{k+1}$ was used in the first step. Therefore, a combination of 
\eqref{convergence-8-1}-\eqref{convergence-8-3} results in 
\begin{eqnarray} 
  I_{6}  \le 
  C_{26} C_7^2 s  \| e^{k+1} \|_{\mathring{H}_{\rm per}^{-1} }^{2/3} \cdot 
  \| \Delta e^{k+1} \|^{4/3}  
  \le C_{27} s \| e^{k+1} \|_{\mathring{H}_{\rm per}^{-1} }^2  
  + \frac{\varepsilon^2}{8} s \| \Delta e^{k+1} \|^2 , 
  \label{convergence-8-4} 
\end{eqnarray} 
with the Young's inequality applied in the last step. 

For the term $I_{7}$, we decompose it into two parts: $I_{7} = I_{7,1} + I_{7,2}$, with 
\begin{eqnarray} 
  I_{7,1} &=& - 12 s  \left( 
  ( \Phi^{k+1} + \phi^{k+1} ) e^{k+1} \nabla \Phi^{k+1} ,  
   \nabla e^{k+1}  \right)  ,  \quad \\
  I_{7,2} &=& - 12 s  \left( ( \phi^{k+1} )^2 \nabla e^{k+1} ,  
  \nabla e^{k+1}  \right)  .  \label{convergence-9-1} 
\end{eqnarray} 
It is clear that the second part is always non-positive: 
\begin{eqnarray} 
  I_{7,2} = - 12 s  \int_\Omega 
  ( \phi^{k+1} )^2  |  \nabla e^{k+1}  |^2  d {\bf x} \le 0 .  \label{convergence-9-2} 
\end{eqnarray} 
For the first part, an application of H\"older inequality shows that 
\begin{eqnarray} 
  I_{7,1}  &\le&  
  C_{28}  s  ( \| \Phi^{k+1} \|_{L^\infty}  + \| \phi^{k+1} \|_{L^\infty}  )    
   \cdot \| \nabla \Phi^{k+1} \|_{L^6} 
   \cdot \| \nabla e^{k+1} \|_{L^{3/2} }    
  \cdot \| e^{k+1} \|_{L^6}    \nonumber 
\\
  &\le& 
  C_{29} C_7^2 s  \| \nabla e^{k+1} \|_{L^{3/2} }    
  \cdot \| e^{k+1} \|_{L^6}  . \label{convergence-9-3} 
\end{eqnarray} 
Again, the $L^\infty$ and $W^{1,6}$ 
bounds \eqref{consistency-bound-1} for the exact and numerical solutions were recalled in the second step. Furthermore, by repeating the same analyses as \eqref{convergence-8-2}-\eqref{convergence-8-3}, we are able to arrive at the following estimate, similar to \eqref{convergence-8-4}: 
\begin{eqnarray} 
  I_{7,1}  \le 
  C_{30} C_7^2 s  
   \cdot \| e^{k+1} \|_{\mathring{H}_{\rm per}^{-1} }^{1/3} \cdot 
  \| \Delta e^{k+1} \|^{2/3}  
  \le C_{31} s  \| e^{k+1} \|_{\mathring{H}_{\rm per}^{-1} }^2  
  + \frac{\varepsilon^2}{8} s  \| \Delta e^{k+1} \|^2 . 
  \label{convergence-9-4} 
\end{eqnarray} 
Consequently, a combination of \eqref{convergence-9-1}), \eqref{convergence-9-2} and \eqref{convergence-9-4} leads to 
\begin{eqnarray} 
  I_{7}  
  \le C_{31}  s  \| e^{k+1} \|_{\mathring{H}_{\rm per}^{-1} }^2  
  + \frac{\varepsilon^2}{8} s  \| \Delta e^{k+1} \|^2 . 
  \label{convergence-9-5} 
\end{eqnarray}

Similarly, the term $I_{8}$ is also decomposed into two parts: 
$I_{8} = I_{8,1} + I_{8,2}$, with 
\begin{eqnarray} 
  I_{8,1} &=& - 8 A s  \left(  ( \nabla ( \Phi^{k+1} + \phi^{k+1} )  
   \cdot \nabla e^{k+1} )  \nabla \Phi^{k+1} , 
   \nabla e^{k+1}  \right)   ,  \nonumber 
\\
  I_{8,2} &=& - 8 A s  \left(  | \nabla \phi^{k+1} |^2 
   \nabla e^{k+1}  , \nabla e^{k+1}  \right) 
   = - 8 A s  \int_\Omega |  \nabla e^{k+1}  |^4   d {\bf x} \le 0    .  \nonumber \label{convergence-10-1} 
\end{eqnarray} 
For the first part $I_{8,1}$, the following estimate is available, in a similar way as \eqref{convergence-9-3}-\eqref{convergence-9-4}: 
\begin{eqnarray} 
  I_{8,1}  &\le&  
  C_{32}  s  (  \| \nabla \Phi^{k+1} \|_{L^6}  + \| \nabla \phi^{k+1} \|_{L^6}  )    
   \cdot \| \nabla \Phi^{k+1} \|_{L^6} 
   \cdot \| \nabla e^{k+1} \|_{L^6 }    
  \cdot \| \nabla e^{k+1} \|
	\nonumber 
	\\
&\le& 
  C_{33}  C_7^2 s  \| \nabla e^{k+1} \|_{L^6 }    
  \cdot \| e^{k+1} \|  
	\nonumber 
	\\
&\le& 
  C_{34}  C_7^2 s  \| \Delta e^{k+1} \|  
  \cdot \| e^{k+1} \|_{\mathring{H}_{\rm per}^{-1} }^{1/3} \cdot 
  \| \Delta e^{k+1} \|^{2/3}   
	\nonumber
	\\
& \le&  C_{35}  C_7^2 s  \| \Delta e^{k+1} \|^{5/3}  
  \cdot \| e^{k+1} \|_{\mathring{H}_{\rm per}^{-1} }^{1/3}  
  \le  C_{36} s  \| e^{k+1} \|_{\mathring{H}_{\rm per}^{-1} }^2  
  + \frac{\varepsilon^2}{8} s  \| \Delta e^{k+1} \|^2 ,   \nonumber
 \label{convergence-10-3} 
\end{eqnarray} 
in which the $W^{1,6}$ bound \eqref{consistency-bound-1} for the exact and numerical solutions was recalled in the second step, the 3-D Sobolev embedding from $H_{\rm per}^2$ into $W^{1,6}$ and the estimate \eqref{convergence-6-2} were used in the third step, and the Young inequality was applied at the last step. Then we arrive at 
\begin{eqnarray} 
   I_{8} = I_{8,1} + I_{8,2}   \le  I_{8,1}  
   \le C_{36} s  \| e^{k+1} \|_{\mathring{H}_{\rm per}^{-1} }^2  
  + \frac{\varepsilon^2}{8} s  \| \Delta e^{k+1} \|^2 .  
 \label{convergence-10-4} 
\end{eqnarray} 

The term $I_{9}$ can be handled in the same way as $I_{8}$. 
We begin with a decomposition $I_{9} = I_{9,1} + I_{9,2}$, with 
\begin{eqnarray} 
  I_{9,1} &=&  8 A s  \left(  ( \nabla ( \Phi^k + \phi^k )  
   \cdot \nabla e^k )  \nabla \Phi^k , 
   \nabla e^{k+1}  \right)   ,  \nonumber 
\\
  I_{9,2} &=&  8 A s  \left(  | \nabla \phi^{k+1} |^2 
   \nabla e^k  , 
   \nabla e^{k+1}  \right) . \nonumber  \label{convergence-11-1} 
\end{eqnarray} 
The following estimates can be carried out: 
\begin{eqnarray} 
  I_{9,1}  &\le&  
  C_{37}  s  (  \| \nabla \Phi^k \|_{L^6}  + \| \nabla \phi^k \|_{L^6}  )    
   \cdot \| \nabla \Phi^k \|_{L^6} 
   \cdot \| \nabla e^k \|_{L^6 }    
  \cdot \| \nabla e^{k+1} \|    \nonumber 
\\
  &\le& 
  C_{38}  C_7^2 s  \| \nabla e^k \|_{L^6 }    
  \cdot \| e^{k+1} \|  \nonumber 
\\
  &\le& 
  C_{39}  C_7^2 s  \| \Delta e^k \| 
  \cdot \| e^{k+1} \|_{\mathring{H}_{\rm per}^{-1} }^{1/3} \cdot 
  \| \Delta e^{k+1} \|^{2/3}   \nonumber   
\\
  &\le& 
   C_{40} s  \| e^{k+1} \|_{\mathring{H}_{\rm per}^{-1} }^2  
  + \frac{\varepsilon^2}{16} s  
  (  \| \Delta e^{k+1} \|^2  +  \| \Delta e^k \|^2  )  ,  \nonumber \label{convergence-11-2} 
\\ 
  I_{9,2}  &\le&  
  C_{41}  s  \| \nabla \phi^k \|_{L^6}^2     
   \cdot \| \nabla e^k \|_{L^6 }   \cdot \| \nabla e^{k+1} \|    
  \le C_{42}  C_7^2 s  \| \nabla e^k \|_{L^6 }    
  \cdot \| e^{k+1} \|  \nonumber 
\\
  &\le& 
  C_{43}  C_7^2 s  \| \Delta e^k \|  
  \cdot \| e^{k+1} \|_{\mathring{H}_{\rm per}^{-1} }^{1/3} \cdot 
  \| \Delta e^{k+1} \|^{2/3}   \nonumber   
\\
  &\le& 
   C_{44} s  \| e^{k+1} \|_{\mathring{H}_{\rm per}^{-1} }^2  
  + \frac{\varepsilon^2}{16} s  
  (  \| \Delta e^{k+1} \|^2  +  \| \Delta e^k \|^2  )  .  \nonumber 
 \label{convergence-11-3} 
\end{eqnarray} 
Consequently, we get 
\begin{eqnarray} 
   I_{9} = I_{9,1} + I_{9,2}   
   \le C_{45}  s  \| e^{k+1} \|_{\mathring{H}_{\rm per}^{-1} }^2  
  + \frac{\varepsilon^2}{8} s  (  \| \Delta e^{k+1} \|^2  +  \| \Delta e^k \|^2  )  .  
 \label{convergence-11-4} 
\end{eqnarray}

Finally, a combination of \eqref{convergence-2}, \eqref{convergence-3-4}, \eqref{convergence-4-1}, \eqref{convergence-5-1}, \eqref{convergence-5-2}, \eqref{convergence-6-5}, \eqref{convergence-7-5},
\eqref{convergence-8-4}, \eqref{convergence-9-5}, \eqref{convergence-10-4} and \eqref{convergence-11-4} yields that 
\begin{eqnarray} 
  && 
    \| e^{k+1} \|_{\mathring{H}_{\rm per}^{-1} }^2 
  - \| e^k \|_{\mathring{H}_{\rm per}^{-1} }^2   
  +  2 ( \varepsilon^{-2}  + \eta )  s \| e^{k+1} \|^2 
  + \frac{9}{8} \varepsilon^2 s \| \Delta e^{k+1}  \|^2  \nonumber 
\\
&\le&   
    C_{46}  s 
  ( \| e^{k+1} \|_{\mathring{H}_{\rm per}^{-1} }^2  
   + \| e^k \|_{\mathring{H}_{\rm per}^{-1} }^2  ) 
   + \frac38  \varepsilon^2 s \| \Delta e^k  \|^2
   + s \| \tau^k \|^2  .   \label{convergence-12} 
\end{eqnarray}
Subsequently, an application of discrete Gronwall inequality leads to an $\ell^\infty (0,T; \mathring{H}_{\rm per}^{-1}) \cap \ell^2 (0,T; H_{\rm per}^2)$ convergence of the numerical scheme \eqref{scheme-CS-Hm1}: 
\begin{eqnarray}  
    \| e^{k} \|_{\mathring{H}_{\rm per}^{-1} }^2   
  + \frac34 \varepsilon^2 s \sum_{l=0}^{k} \| \Delta e^l  \|^2  
  \le C s^2 ,   \label{convergence-13} 
\end{eqnarray}
for any $1 \le k \le K$. Note that the constant $C$ depends on the exact solution, the physical parameter $\varepsilon$, and final time $T$, independent on $s$. The proof of Theorem \ref{convergence} is finished. 

	\section{Finite difference spatial discretization in 2D}
	\label{sec:fdm}

	\subsection{Notation}
	\label{sec:notation}
In this subsection we define the discrete spatial difference operators, function space, inner products and norms, following the notations used in \cite{feng2016preconditioned, wang2011energy, wise09a}.  Let $\Omega = (0,L_x)\times(0,L_y)$, where, for simplicity, we assume $L_x =L_y =: L > 0$. We write $L = m\cdot h$, where $m$ is a positive integer. The parameter $h = \frac{L}{m}$ is called the mesh or grid spacing. We define the following two uniform, infinite grids with grid spacing $h>0$:
	\[
E := \{ x_{i+\hf} \ |\ i\in {\mathbb{Z}}\}, \quad C := \{ x_i \ |\ i\in {\mathbb{Z}}\},
	\]
where $x_i = x(i) := (i-\hf)\cdot h$. Consider the following 2D discrete periodic function spaces: 
	\begin{eqnarray*}
{\mathcal V}_{\rm per} &:=& \left\{\nu: E\times E\rightarrow {\mathbb{R}}\ \middle| \ \nu_{i+\frac12,j+\frac12}= \nu_{i+\frac12+\alpha m,j+\frac12+\beta m}, \ \forall \, i,j,\alpha,\beta\in \mathbb{Z}  \right\},
	\\
{\mathcal C}_{\rm per} &:=& \left\{\nu: C\times C
\rightarrow {\mathbb{R}}\ \middle| \ \nu_{i,j} = \nu_{i+\alpha m,j+\beta m}, \ \forall \, i,j,\alpha,\beta\in \mathbb{Z} \right\},
	\\
{\mathcal E}^{\rm ew}_{\rm per} &:=& \left\{\nu: E\times C\rightarrow {\mathbb{R}}\ \middle| \ \nu_{i+\frac12,j}= \nu_{i+\frac12+\alpha m,j+\beta m}, \ \forall \,  i,j,\alpha,\beta\in \mathbb{Z}  \right\},
	\\
{\mathcal E}^{\rm ns}_{\rm per} &:=& \left\{\nu: C \times E\rightarrow {\mathbb{ R}}\ \middle| \ \nu_{i,j+\frac12}= \nu_{i+\alpha m,j+\frac12+\beta m}, \ \forall \, i,j,\alpha,\beta\in \mathbb{Z}  \right\}.
	\end{eqnarray*}		
The functions of ${\mathcal V}_{\rm per}$ are called {\emph{vertex centered functions}}; those of ${\mathcal C}_{\rm per}$ are called {\emph{cell centered functions}}. The functions of ${\mathcal E}^{\rm ew}_{\rm per}$ are called {\emph{east-west edge-centered functions}}, and the functions of ${\mathcal E}^{\rm ns}_{\rm per}$ are called {\emph{north-south edge-centered functions}}.  We also define the mean zero space 
	\[
\mathring{\mathcal C}_{\rm per}:=\left\{\nu\in {\mathcal C}_{\rm per} \ \middle| \overline{\nu} :=  \frac{h^2}{| \Omega|} \sum_{i,j=1}^m \nu_{i,j}  = 0\right\} .
	\]

We now introduce the important difference and average operators on the spaces:  
	\begin{eqnarray*}
&& A_x \nu_{i+\hf,\Box} := \frac{1}{2}\left(\nu_{i+1,\Box} + \nu_{i,\Box} \right), \quad D_x \nu_{i+\hf,\Box} := \frac{1}{h}\left(\nu_{i+1,\Box} - \nu_{i,\Box} \right),\\
&& A_y \nu_{\Box,i+\hf} := \frac{1}{2}\left(\nu_{\Box,i+1} + \nu_{\Box,i} \right), \quad D_y \nu_{\Box,i+\hf} := \frac{1}{h}\left(\nu_{\Box,i+1} - \nu_{\Box,i} \right) , 
	\end{eqnarray*}
with $A_x,\, D_x: {\mathcal C}_{\rm per}\rightarrow{\mathcal E}_{\rm per}^{\rm ew}$ if $\Box$ is an integer, and $A_x,\, D_x: {\mathcal E}^{\rm ns}_{\rm per}\rightarrow{\mathcal V}_{\rm per}$ if $\Box$ is a half-integer, with $A_y,\, D_y: {\mathcal C}_{\rm per}\rightarrow{\mathcal E}_{\rm per}^{\rm ns}$ if $\Box$ is an integer, and $A_y,\, D_y: {\mathcal E}^{\rm ew}_{\rm per}\rightarrow{\mathcal V}_{\rm per}$ if $\Box$ is a half-integer. Likewise,
	\begin{eqnarray*}
&&a_x \nu_{i,\Box} := \frac{1}{2}\left(\nu_{i+\hf,\Box} + \nu_{i-\hf,\Box} \right),	 \quad d_x \nu_{i,\Box} := \frac{1}{h}\left(\nu_{i+\hf,\Box} - \nu_{i-\hf,\Box} \right),
	\\
&&a_y \nu_{\Box,j} := \frac{1}{2}\left(\nu_{\Box,j+\hf} + \nu_{\Box,j-\hf} \right),	 \quad d_y \nu_{\Box,j} := \frac{1}{h}\left(\nu_{\Box,j+\hf} - \nu_{\Box,j-\hf} \right),
	\end{eqnarray*}
with $a_x,\, d_x : {\mathcal E}_{\rm per}^{\rm ew}\rightarrow{\mathcal C}_{\rm per}$ if $\Box$ is an integer, and $a_x,\ d_x: {\mathcal V}_{\rm per}\rightarrow{\mathcal E}^{\rm ns}_{\rm per}$ if $\Box$ is a half-integer; and with $a_y,\, d_y : {\mathcal E}_{\rm per}^{\rm ns}\rightarrow{\mathcal C}_{\rm per}$ if $\Box$ is an integer, and $a_y,\ d_y: {\mathcal V}_{\rm per}\rightarrow{\mathcal E}^{\rm ew}_{\rm per}$ if $\Box$ is a half-integer.
    
Define the 2D center-to-vertex derivatives $\mD_x,\, \mD_y : {\mathcal C}_{\rm per}\rightarrow{\mathcal V}_{\rm per}$ component-wise as
	\begin{eqnarray*}
\mD_x \nu_{i+\hf,j+\hf} &:=& A_y(D_x\nu)_{i+\hf,j+\hf} = D_x(A_y\nu)_{i+\hf,j+\hf}
	\nonumber
	\\
&=& \frac{1}{2h}\left(\nu_{i+1,j+1}-\nu_{i,j+1}+\nu_{i+1,j}-\nu_{i,j} \right) ,
	\\
\mD_y\nu_{i+\hf,j+\hf} &:=& A_x(D_y\nu)_{i+\hf,j+\hf} = D_y(A_x\nu)_{i+\hf,j+\hf} 
	\nonumber
	\\
&=& \frac{1}{2h}\left(\nu_{i+1,j+1}-\nu_{i+1,j}+\nu_{i,j+1}-\nu_{i,j} \right) . 
	\end{eqnarray*}		

The utility of these definitions is that the differences $\mD_x$ and $\mD_y$ are collocated on the grid, unlike the case for $D_x$, $D_y$. Define the 2D vertex-to-center derivatives $\md_x,\, \md_y  : {\mathcal V}_{\rm per}\rightarrow{\mathcal C}_{\rm per}$  component-wise as
	\begin{eqnarray*}
\md_x \nu_{i,j} &:=& a_y(d_x \nu)_{i,j} = d_x(a_y \nu)_{i,j} 
	\nonumber
	\\
&=& \frac{1}{2h}\left(\nu_{i+\hf,j+\hf}- \nu_{i-\hf,j+\hf}+\nu_{i+\hf,j-\hf}-\nu_{i-\hf,j-\hf}\right) ,
	\\
\md_y \nu_{i,j} &:=& a_x(d_y \nu)_{i,j} = d_y(a_x \nu)_{i,j}
	\nonumber
	\\
&=& \frac{1}{2h}\left(\nu_{i+\hf,j+\hf}- \nu_{i+\hf,j-\hf}+\nu_{i-\hf,j+\hf}-\nu_{i-\hf,j-\hf}\right).
	\end{eqnarray*}

Now the discrete gradient operator, $\nabla^v_h$: ${\mathcal C}_{\rm per}\rightarrow {\mathcal V}_{\rm per}\times {\mathcal V}_{\rm per}$, becomes  
	\[
\nabla^v_h \nu_{i+\hf,j+\hf} := (\mD_x \nu_{i+\hf,j+\hf}, \mD_y \nu_{i+\hf,j+\hf}). 
	\]	
The standard 2D discrete Laplacian, $\Delta_h : {\mathcal C}_{\rm per}\rightarrow{\mathcal C}_{\rm per}$, is given by 
	\[
\Delta_h \nu_{i,j} := d_x(D_x \nu)_{i,j} + d_y(D_y \nu)_{i,j} = \frac{1}{h^2}\left( \nu_{i+1,j}+\nu_{i-1,j}+\nu_{i,j+1}+\nu_{i,j-1} - 4\nu_{i,j}\right).
	\]
The 2D vertex-to-center average, $\mathfrak{A} : {\mathcal V}_{\rm per}\rightarrow{\mathcal C}_{\rm per}$, is defined  to be 
	\[
\mathfrak{A} \nu_{i,j} := \frac{1}{4}\left( \nu_{i-\hf,j-\hf}+\nu_{i-\hf,j+\hf}+\nu_{i+\hf,j+\hf}+\nu_{i+\hf,j-\hf}\right),
	\]
and the  2D center-to-vertex average, $\mathfrak{a} : {\mathcal C}_{\rm per} \rightarrow {\mathcal V}_{\rm per}$, becomes  
	\[
\mathfrak{a} \nu_{i+\hf,j+\hf} := \frac{1}{4}\left( \nu_{i,j}+\nu_{i+1,j}+\nu_{i,j+1}+\nu_{i+1,j+1}\right) . 
	\]		
The 2D \emph{skew} Laplacian, $\Delta^v_h : {\mathcal C}_{\rm per}\rightarrow{\mathcal C}_{\rm per}$, is introduced as
	\begin{eqnarray*}
\Delta^v_{h}\nu_{i,j} &=& \md_x(\mD_x\nu)_{i,j} + \md_y( \mD_y\nu)_{i,j} 
=\frac{1}{2h^2}\left(\nu_{i+1,j+1}+\nu_{i-1,j+1}+\nu_{i+1,j-1}+\nu_{i-1,j-1} - 4\nu_{i,j}\right) .
	\end{eqnarray*}		
In addition, the 2D undivided laplacian operator for non-constant mobility is 
	\begin{eqnarray*}
\nabla_h^v \cdot \left( {\cal M}^v(\nu) \nabla_h^v \nu \right)_{ij} := \md_x({\cal M}^v(\nu)\, \mD_x\nu )_{i,j}+ \md_y({\cal M}^v(\nu)\, \mD_y\nu )_{i,j}, ~~{\cal M}^v(\nu)= \left(\mathfrak{a}{\cal M}(\nu)\right)_{i+\hf,j+\hf}
	\end{eqnarray*}	
Hence, the 2D discrete p-Laplacian operator turns out to be   
	\begin{eqnarray*}
\nabla_h^v \cdot \left( \left| \nabla_h^v \nu\right|^{p-2} \nabla_h^v \nu \right)_{ij} := \md_x(r\, \mD_x\nu )_{i,j}+ \md_y(r\, \mD_y\nu )_{i,j},
	\end{eqnarray*}	
with 
	\[
r_{i+\frac{1}{2},j+\frac{1}{2}}:=\left[(\mD_x u )_{i+\frac{1}{2},j+\frac{1}{2}}^2+(\mD_y u )_{i+\frac{1}{2},j+\frac{1}{2}}^2\right]^{\frac{p-2}{2}}.
	\]	
Clearly, for $p=2$, we have $\Delta^v_{h}\nu = \nabla_h^v \cdot \left( \left| \nabla_h^v \nu\right|^{p-2} \nabla_h^v \nu \right)$.

Now we are ready to define the following grid inner products:  
	\begin{eqnarray*}
\ciptwo{\nu}{\xi}_2 &:= h^2\sum_{i=1}^m\sum_{j=1}^n \nu_{i,j}\psi_{i,j},\quad \nu,\, \xi\in {\mathcal C}_{\rm per},&\quad
\viptwo{\nu}{\xi} := \ciptwo{\mathfrak{A}(\nu\xi)}{1}_2 ,\quad \nu,\, \xi\in{\mathcal V}_{\rm per},
\\
\eipew{\nu}{\xi} &:= \ciptwo{A_x(\nu\xi)}{1}_2 ,\quad \nu,\, \xi\in{\mathcal E}^{\rm ew}_{\rm per},
&\quad 
\eipns{\nu}{\xi} := \ciptwo{A_y(\nu\xi)}{1}_2 ,\quad \nu,\, \xi\in{\mathcal E}^{\rm ns}_{\rm per}.
	\end{eqnarray*}	
Suppose that $\zeta\in\mathring{\mathcal C}_{\rm per}$, then there is a unique solution $\msfT_h[\zeta]\in\mathring{\mathcal C}_{\rm per}$ such that $-\Delta_h \msfT_h[\zeta] = \zeta$. We often write, in this case, $\msfT_h[\zeta] = -\Delta^{-1}_h \zeta$. The discrete analog of the $\mathring{H}^{-1}_{\rm per}$ inner product is defined as 
	\[
\moneinn{\zeta,\xi} := \iprd{\zeta}{\msfT_h[\xi]}_2 = \iprd{\msfT_h[\zeta]}{\xi}_2,\quad \zeta,\, \xi\in\mathring{\mathcal C}_{\rm per}.
	\]
where summation-by-parts formulae \cite{diegel2015analysis,wise09a} guarantees the symmetry and the second equality.

We now define the following norms for cell-centered functions.  If $\nu\in\mathring{\mathcal C}_{\rm per}$, then $\monenrm{\nu}^2 = \moneinn{\nu,\nu}$. If $\nu\in {\mathcal C}_{\rm per}$, then $\nrm{\nu}_2^2 := \ciptwo{\nu}{\nu}_2$; $\nrm{\nu}_p^p := \ciptwo{|\nu|^p}{1}_2$ ($1\le p< \infty$), and $\nrm{\nu}_\infty := \max_{1\le i\le m \atop 1\le j\le n}\left|\nu_{i,j}\right|$.
Similarly, we define the gradient norms: for $\nu\in{\mathcal C}_{\rm per}$,
	\[
\nrm{\nabh^v\nu}_p^p := \langle |\nabla_h^v\nu|^p, 1\rangle, \quad |\nabh^v\nu|^p:=[(\mD_x\nu)^2 +(\mD_y\nu)^2]^{\frac{p}{2}} = \left[\nabla_h^v\nu\cdot\nabla_h^v\nu  \right]^{\frac{p}{2}}  \in \mathcal{V}_{\rm per}, \quad 2\le p < \infty, 
	\]
and
	\[
\nrm{ \nabla_h \nu}_2^2 : = \eipew{D_x\nu}{D_x\nu} + \eipns{D_y\nu}{D_y\nu} .
	\]

\subsection{Fully discrete finite difference scheme}	
With the machinery in last subsection, the discrete energy of FCH can be rewritten as: 
\begin{eqnarray}
{\cal F}_h (\phi) = {\cal F}_{c,h} (\phi)- {\cal F}_{e,h} (\phi)
\label{dis-energy-FCH}
\end{eqnarray} 
where
	\begin{eqnarray}
{\cal F}_{c,h} (\phi)  &=&  \frac{\varepsilon^{-2} }{2}  \nrm{ \phi }_{6}^6   + \left( \frac{\varepsilon^{-2} }{2} + \frac{\eta}{2} \right) \nrm{ \phi }_2^2  + \frac{\varepsilon^2}{2} \nrm{ \Delta_h \phi }_2^2+{\cal H}_{h} (\phi) , 
	\\
{\cal F}_{e,h} (\phi)  &=&  \left( \varepsilon^{-2} + \frac{\eta}{4} \right) \nrm{ \phi }_{4}^4+\left( 1 + \frac{\eta \varepsilon^2}{2} \right)  \nrm{ \nabla_h^v \phi }_2^2  + A \nrm{ \phi }_{4}^4 + A \nrm{ \nabla_h^v\phi }_{4}^4,   
	\end{eqnarray}
and
	\begin{eqnarray}
{\cal H}_{h} (\phi)  &=& A \nrm{ \phi }_{4}^4 + A \nrm{ \nabla_h^v\phi }_{4}^4 + 3 \left(\phi^2, \mathfrak{A}(\left| \nabla_h^v \phi \right|^2)\right)_2.  
	\end{eqnarray}

\begin{prop}\label{prop:1st}
Suppose $\phi \in {\mathcal C}_{\rm per}$. The first variational derivative of ${\cal H}_h(\phi)$ is
	\begin{eqnarray*}
{\delta\cal H}_{h}(\phi)
&=&  4A\phi^3-
4A\left(\md_x\left([(\mD_x\phi)^2 +(\mD_y\phi)^2] \mD_x\phi\right)
+ \md_y\left([(\mD_x\phi)^2 +(\mD_y\phi)^2] \mD_y\phi\right)\right) \\
& &+ \ 6 \phi \mathfrak{A}[(\mD_x\phi)^2 +(\mD_y\phi)^2]
-
6 \left(\md_x\left(\mathfrak{a}\left(\phi^2\right)\mD_x\phi\right)
+ \md_y\left(\mathfrak{a}\left(\phi^2\right)\mD_y\phi\right) \right).
	\end{eqnarray*}
	\end{prop}
	
	\begin{lem}\label{lem:disconv}
Suppose that $\phi\in \mathcal{C}_{\rm per}$ and $A\geq 1$ then  ${\cal H}_{h} (\phi)$, ${\cal F}_{c,h} (\phi)$ and ${\cal F}_{e,h} (\phi)$ are strictly convex.
	\end{lem}
	\begin{proof}
The convexity proof of ${\cal H}_h(\phi)$ is similar to Lemma~\ref{lem:Hconv}. The convexities of  ${\cal F}_{c,h} (\phi)$ and ${\cal F}_{e,h} (\phi)$ follow from the convexity of $\mathcal{H}_h(\phi)$.
	\end{proof}

According to Proposition~\ref{prop:1st} and some other standard calculations~\cite{shen12}, the fully discretized finite difference convex splitting  scheme can be rewritten as: given $f, g \in \mathcal{C}_{\rm per}$, find $\phi^{k+1}, \tilde\mu^{k+1} \in \mathcal{C}_{\rm per}$ such that
\begin{eqnarray} 
\phi^{k+1} - s \Delta_h \tilde\mu^{k+1} =  g, \label{scheme-FD-1}
\end{eqnarray}
where
	\begin{eqnarray}
\tilde\mu^{k+1}&=& \delta_\phi {\cal F}_{c,h}(\phi^{k+1}) - \delta_\phi {\cal F}_{e,h}(\phi^k) \nonumber\\
 &=& 3 \varepsilon^{-2} (\phi^{k+1})^5 
+ 4 A  (\phi^{k+1})^3 
+ (\varepsilon^{-2}+\eta) \phi^{k+1}
+ 6 (\phi^{k+1})^2 \mathfrak{A}(|\nabla_h^v\phi^{k+1}|^2)  + \varepsilon^2 \Delta_h^2 \phi ^{k+1}
	\nonumber 
	\\
& & - \ 6 \nabla_h^v \cdot (\mathfrak{a}\big(\big(\phi^{k+1}\big)^2\big)\nabla_h^v \phi^{k+1} )
- 4A\nabla_h^v \cdot ( | \nabla_h^v \phi^{k+1} |^{2} \nabla_h^v \phi^{k+1} )-f ,  	\label{scheme-FD-2} 
	\end{eqnarray} 
with
	\begin{eqnarray}
g := \phi^k,~ f:= - ( 4 \varepsilon^{-2} +  \eta ) ( \phi^k )^3
  + ( 2 + \eta \varepsilon^2 )  \Delta_h^v \phi^k  
  - 4 A ( \phi^k )^3     
  + 4 A \nabla_h^v \cdot ( | \nabla_h^v \phi^k |^{2} \nabla_h^v \phi^k ) .  
	\label{scheme-FD-3} 
	\end{eqnarray}   
This scheme is mass-conservative in the sense that $  \phi-g \in \mathring{\mathcal C}_{\rm per}$. 
	\begin{thm}
	\label{thm:eng-decay-FD} 
The fully discrete scheme (\ref{scheme-FD-1}) -- (\ref{scheme-FD-3}) is unconditionally discrete energy  stable, ${\cal F}_h (\phi^{k+1}) \le {\cal F}_h (\phi^k)$, and unconditionally uniquely solvable.
\end{thm}
\begin{proof}
The proof follows from Lemma~\ref{lem:disconv} and the discrete version of \eqref{splitting-ineq} found in~\cite{wise09a}.
	\end{proof}

Following similar ideas as in the analyses for the semi-discrete case, we are able to derive the unique solvability, unconditional energy stability and the $\ell^\infty (0,T; H^{-1}) \cap \ell^2 (0,T; H^2)$ convergence for the fully discrete scheme (\ref{scheme-FD-1}) -- (\ref{scheme-FD-3}).  The detailed proofs are skipped for the sake of brevity and are left to interested readers. 


	\begin{thm} \label{convergence-FD}
Let $\Phi \in \mathcal{R}_2$ (see (\ref{regularity assumption-2})) be the exact periodic solution of the FCH equation \eqref{equation-FCH-Hm1} with the initial data $\Phi(0) =\phi_0\in H^2_{\rm per}(\Omega)$. Suppose $\phi$ is the fully-discrete solution of \eqref{scheme-FD-1} -- \eqref{scheme-FD-3}. Then the following convergence result holds as $s$, $h$ goes to zero: 
\begin{eqnarray} 
  \| \Phi (t_k) - \phi^k \|_{-1}   
  + \left( \varepsilon^2 s \sum_{\ell=0}^{k} \| \Delta_h  ( \Phi (t_\ell) - \phi^\ell ) \|^2  \right)^{1/2} 
  \le C  ( s + h^2 )  ,  
\label{convergence-FD-1}
\end{eqnarray}     
where the constant $C>0$ is independent of $s$ and $h$. 
	\end{thm}

	\section{Preconditioned steepest descent (PSD) solver}\label{sec:psd}

In this section we describe a preconditioned steepest descent (PSD) algorithm for advancing the convex splitting scheme in time following the practical and  theoretical framework in~\cite{feng2016preconditioned}. The fully discrete scheme \eqref{scheme-FD-1} -- \eqref{scheme-FD-3} can be recast as a minimization problem with an energy that involves the $\monenrm{\, \cdot\, }^2$ norm: For any $\phi \in \mathcal{C}_{\rm per}$,
	\begin{eqnarray}
E_h[\phi]&=& \frac{1}{2}\nrm{\phi-g}_{-1}^2 
+ \frac{s\varepsilon^{-2}}{2} \norm{\phi}{6}^6
+ \frac{s(\varepsilon^{-2}+\eta)}{2} \norm{\phi}{2}^2\nonumber\\
&&+ As \norm{\phi}{4}^4
+ As\norm{\nabla_h^v u}{4}^4
+ 3\iprd{\phi^2}{\mathfrak{A}\left(|\nabla_h^v\phi|^2\right)}_2
+\frac{s\varepsilon^2}{2}\nrm{\Delta_h \phi}_2^2
+ s\iprd{g}{\phi}_2,  \label{discrete energy-1}
	\end{eqnarray}
which is strictly convex provided that $A\geq 1$. One will observe that the fully discrete scheme \eqref{scheme-FD-1} -- \eqref{scheme-FD-3} is the discrete variation of the strictly convex energy \eqref{discrete energy-1} set equal to zero. The nonlinear scheme at a fixed time level may be expressed as
	\begin{equation}
\mathcal{N}_h[\phi] = f,
	\end{equation}
where
	\begin{eqnarray} 
\mathcal{N}_h[\phi] &=& -\Delta_h^{-1}(\phi-g) + 3s \varepsilon^{-2} \phi^5 + 4 sA  \phi^3 + s(\varepsilon^{-2}+\eta) \phi + 6s \phi^2 \mathfrak{A}(|\nabla_h^v\phi|^2)  	\nonumber 
	\\
&& - 6 s \nabla_h^v \cdot \left(\mathfrak{a}\left(\phi^2\right) \nabla_h^v \phi \right) - 4sA  \nabla_h^v \cdot ( | \nabla_h^v \phi |^{2} \nabla_h^v \phi) + s\varepsilon^2 \Delta_h^2 \phi.
	\label{non-operator} 
	\end{eqnarray} 
The main idea of the PSD solver is to use a linearized version of the nonlinear operator as a pre-conditioner, or in other words, as a metric for choosing the search direction.  A linearized version of the nonlinear operator $\mathcal{N}$  is defined as follows: $\mathcal{L}_h: \mathring{\mathcal C}_{\rm per} \to \mathring{\mathcal C}_{\rm per}$,
	\begin{equation*}
{\mathcal L}_h[\psi] := -\Delta_h^{-1} \psi +s(4\varepsilon^{-2}+\eta+4A+6)\psi - s(6+4A)\Delta_h \psi  +s\varepsilon^2 \Delta_h^2 \psi.
	\end{equation*}
Clearly, this is a positive, symmetric operator, and we use this as a pre-conditioner for the method. Specifically, this ``metric" is used to find an appropriate search directtion for our steepest descent solver~\cite{feng2016preconditioned}. Given the current iterate $\phi^{n}\in {\mathcal C}_{\rm per}$, we define the following \emph{search direction} problem: find $d^n \in \mathring{\mathcal C}_{\rm per}$ such that
\[
{\mathcal L}_h[d^n]= f-\mathcal{N}_h[\phi^n]:=r^n,
\]
where $r^n$ is the nonlinear residual of the $n^{\rm th}$ iterate $\phi^n$. This last equation  can be solved efficiently using the Fast Fourier Transform (FFT).

We then define the next iterate as
	\begin{equation}
\phi^{n+1} = \phi^{n} + \overline{\alpha} d^n,
	\end{equation}
where $\overline{\alpha}\in\mathbb{R}$ is the unique solution to the steepest descent line minimization problem 
	\begin{equation}
\overline{\alpha} := \operatorname*{argmax}_{\alpha\in\mathbb{R}} E_h[\phi^{n} + \alpha d^n]= \operatorname*{argzero}_{\alpha\in\mathbb{R}}\delta E_h[\phi^{n} + \alpha d^n](d^n) .
	\label{eqn-search}
	\end{equation}
The theory in~\cite{feng2016preconditioned} suggests that $\phi^n \to \phi^{k+1}$ geometrically as $n\to \infty$, where $\mathcal{N}_h[\phi^{k+1}] = f$, \emph{i.e.}, $\phi^{k+1}$ is the solution of the scheme \eqref{scheme-FD-1} -- \eqref{scheme-FD-3} at time level $k+1$. Furthermore, the convergence rate is independent of $h$. 
	
	\section{Numerical results}\label{sec:num}
	 
We perform some numerical experiments with the PSD solver to support the theoretical results in previous sections. The finite difference search direction equations and Poisson equations are solved efficiently using the Fast Fourier Transform (FFT). Though we do not present it here, we can also implement the scheme by using the pseudo-spectral method for spatial discretization~\cite{boyd2001chebyshev, cheng2015fourier, feng2016preconditioned,hesthaven2007spectral}. 

	\subsection{Convergence test}
	\label{subsec-convergence}

In this numerical experiment, we apply the benchmark problem in  \cite{Christlieb14high,jones2013development} to show that our scheme is first order accurate in time. The convergence test is performed with the initial data given by 
	\begin{eqnarray}
\phi(x,y,0)=2\exp\left[\sin(\frac{2\pi x}{L_x})+\sin(\frac{2\pi y}{L_y})-2\right]+2.2\exp\left[-\sin(\frac{2\pi x}{L_x})-\sin(\frac{2\pi y}{L_y})-2\right]-1.
\label{eqn:init}
	\end{eqnarray}
We use a quadratic refinement path, \emph{i.e.}, $s=Ch^2$. At the final time $T=0.32$, we expect the global error to be $\mathcal{O}(s)+\mathcal{O}(h^2)=\mathcal{O}(h^2)$ in either the  $\ell^2$ or $\ell^\infty$ norm, as $h, s\to 0$.  Since an exact solution is not available, instead of calculating the error at the final time, we compute the Cauchy difference, which is defined as $\delta_\phi: =\phi_{h_f}-\mathcal{I}_c^f(\phi_{h_c})$, where $\mathcal{I}_c^f$ is a bilinear interpolation operator. This requires having a relatively coarse solution, parametrized by $h_c$, and a relatively fine solution, parametrized by $h_f$, where $h_c = 2 h_f$, at the same final time. The Cauchy difference is also expected to be $\mathcal{O}(s)+\mathcal{O}(h^2)=\mathcal{O}(h^2)$, as $h, s\to 0$. The other parameters are given by $L_x=L_y=3.2$, $\varepsilon=0.18$, $A=1.0$, $\eta=1.0$, $s=0.1h^2$. The norms of Cauchy difference, the convergence rates, average iteration number and average CPU time (in seconds) can be found in Table~\ref{tab:cov}. The  results confirm our expectation for the convergence order and also demonstrate the efficiency of our algorithm. Moreover, the semi-log scale of the residual $\nrm{r^n}_{\infty}$ with respect to the PSD iterations can be found in Fig. \ref{fig:complexity}, which confirms the expected geometric convergence rate of the PSD solver predicted by the theory in~\cite{feng2016preconditioned}.
 
\begin{table}[!htb]
	\begin{center}
		\caption{Errors, convergence rates, average iteration numbers and average CPU time (in seconds) for each time step. Parameters are given in the
		text, and the initial data is defined in \eqref{eqn:init}. The refinement path is $s=0.1h^2$. } \label{tab:cov}
		\begin{tabular}{ccccccc}
			\hline $h_c$&$h_{f}$&$\nrm{\delta_\phi}_{2}$ & Rate&$\#_{iter}$ &$T_{cpu}(h_f)$\\
			\hline $\frac{3.2}{16}$&$\frac{3.2}{32}$& $1.8131\times 10^{-2}$&- & 27  & 0.0136
			\\$\frac{3.2}{32}$&$\frac{3.2}{64}$& $4.2725\times 10^{-3}$ &2.09& 25 &0.0493 
			\\ $\frac{3.2}{64}$ &$\frac{3.2}{128}$& $7.7211\times 10^{-4}$&2.47 & 19 &0.1534& 
			\\ $\frac{3.2}{128}$ & $\frac{3.2}{256}$ &$1.7075\times 10^{-4}$&2.18 &11 &0.4809
			\\ $\frac{3.2}{256}$ & $\frac{3.2}{512}$&$4.0134\times 10^{-5}$&2.09&05 &2.1579\\
			\hline
		\end{tabular}
	\end{center}
\end{table}

\begin{figure}[h]
	\begin{center}
	\includegraphics[height=0.6\textwidth]{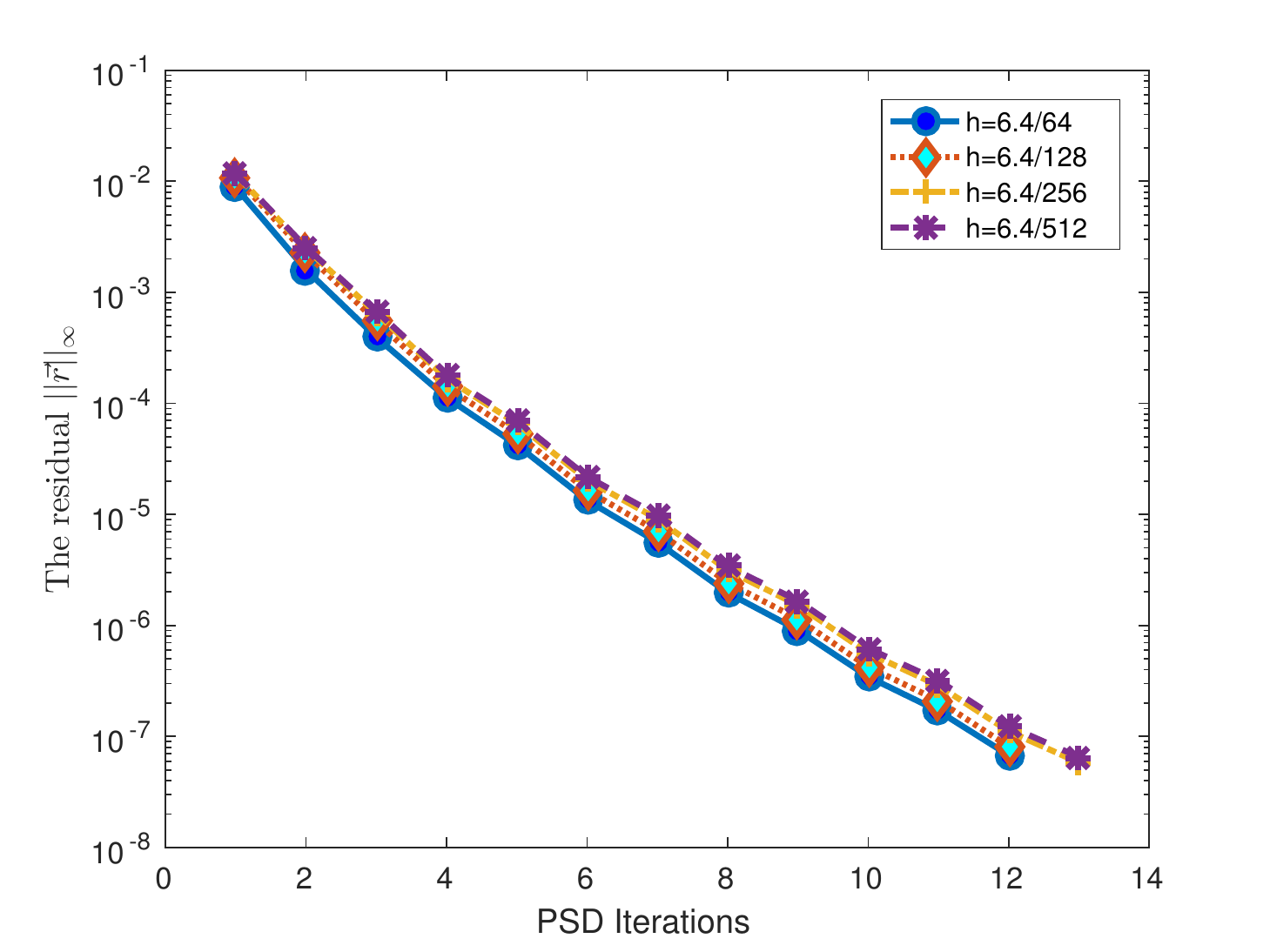} 
	\caption{Solver convergence (complexity) test for the problem defined in Section~\ref{subsec-convergence}.  The only difference is that for this test, we use a fixed time step size, $s=1.0\times 10^{-5}$ for all runs. We plot on a semi-log scale of the residual $\nrm{r^n}_{\infty}$ with respect to the PSD iteration count $n$ at the 20th time step, \emph{i.e.}, $t=2.0\times 10^{-4}$. The initial data is defined in \eqref{eqn:init}, $L_x=L_y=6.4$, $\varepsilon=0.18$, $A=1.0$, $\eta=1.0$, and the grid sizes are as specified in the legend. We observe that the residual is decreasing by a nearly constant factor for each iteration.} 
	\label{fig:complexity}
	\end{center}
\end{figure}

\subsection{Long time simulation of benchmark problem}
Time snapshots of the benchmark problem in  \cite{Christlieb14high,jones2013development} for the long time test can be found in Fig.~\ref{fig:long-time-fch}.  The initial data is defined in \eqref{eqn:init} and  the other parameters are given by $L_x=L_y=6.4$, $\varepsilon=0.18$, $A=1.0$, $\eta=1.0$, $s=1\times 10^{-4}$ and $h={6.4}/{256}$. The numerical results in Fig.~\ref{fig:long-time-fch} are consistent with earlier work on this topic in \cite{Christlieb14high,jones2013development}.
\begin{figure}[h]
	\begin{center}
		\begin{subfigure}{0.48\textwidth}
			\includegraphics[height=0.48\textwidth,width=0.48\textwidth]{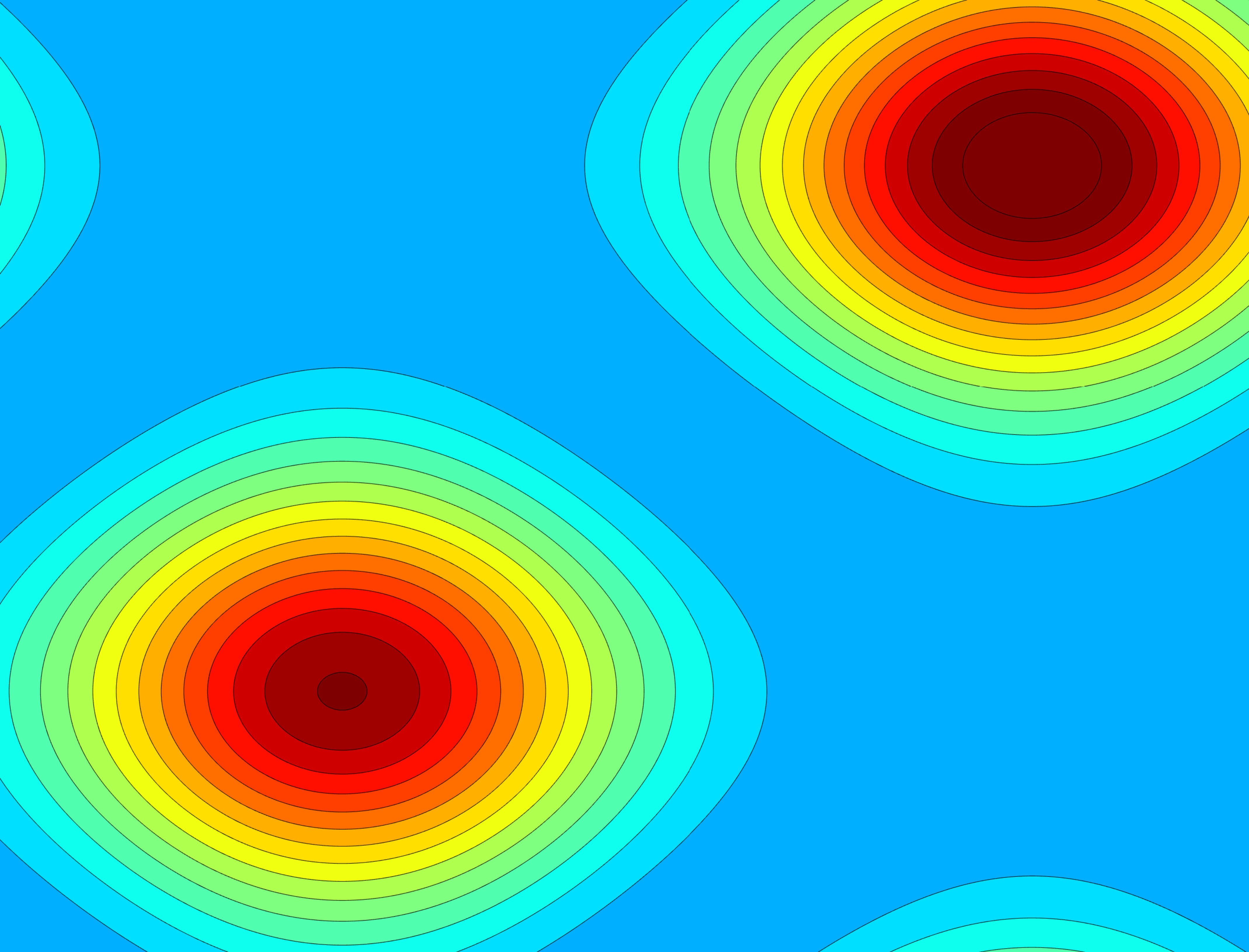} 
			\includegraphics[height=0.48\textwidth,width=0.48\textwidth]{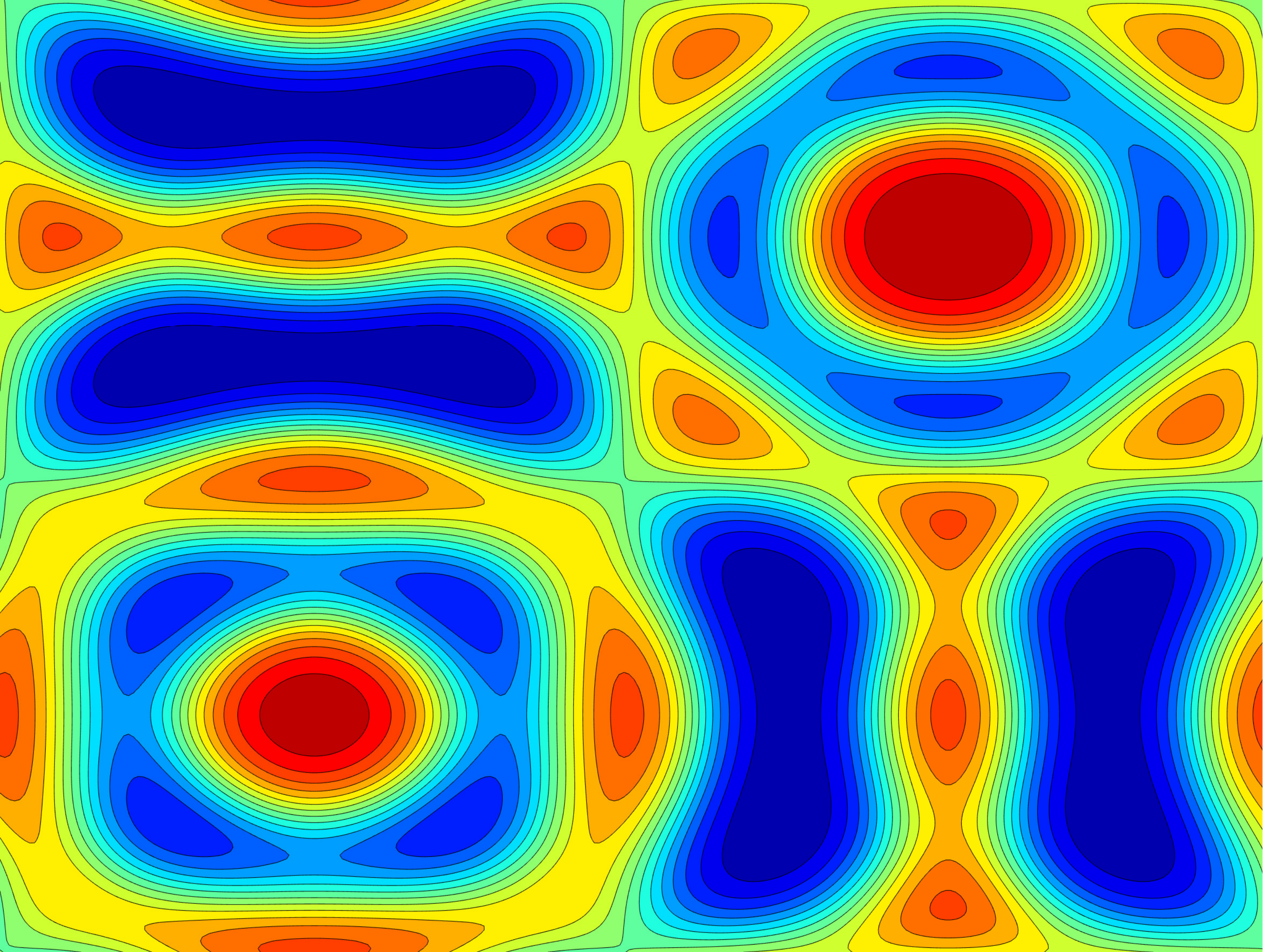} 
			\caption*{$t=0, 0.2$}
		\end{subfigure}
		\begin{subfigure}{0.48\textwidth}
			\includegraphics[height=0.48\textwidth,width=0.48\textwidth]{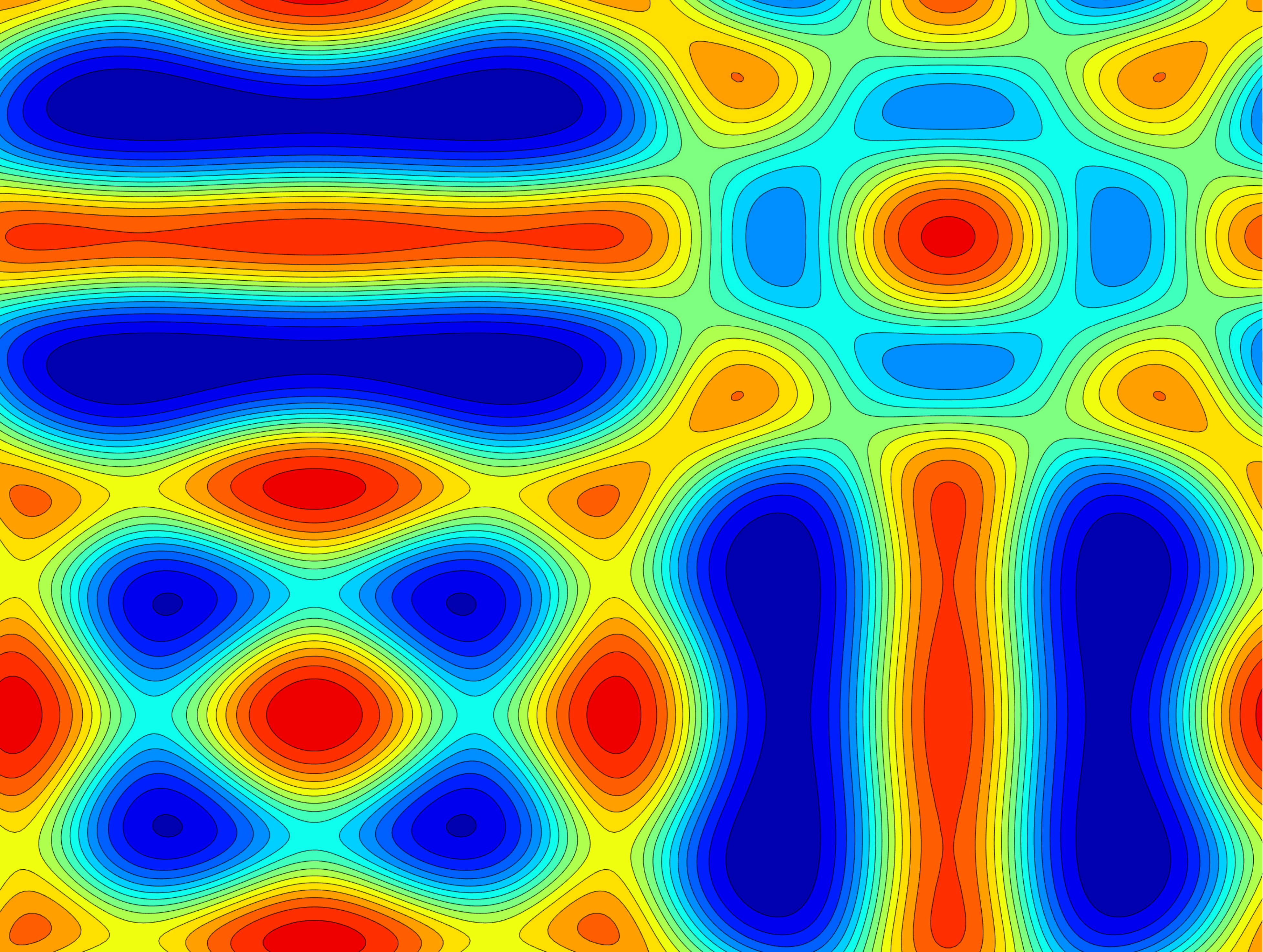} 
			\includegraphics[height=0.48\textwidth,width=0.48\textwidth]{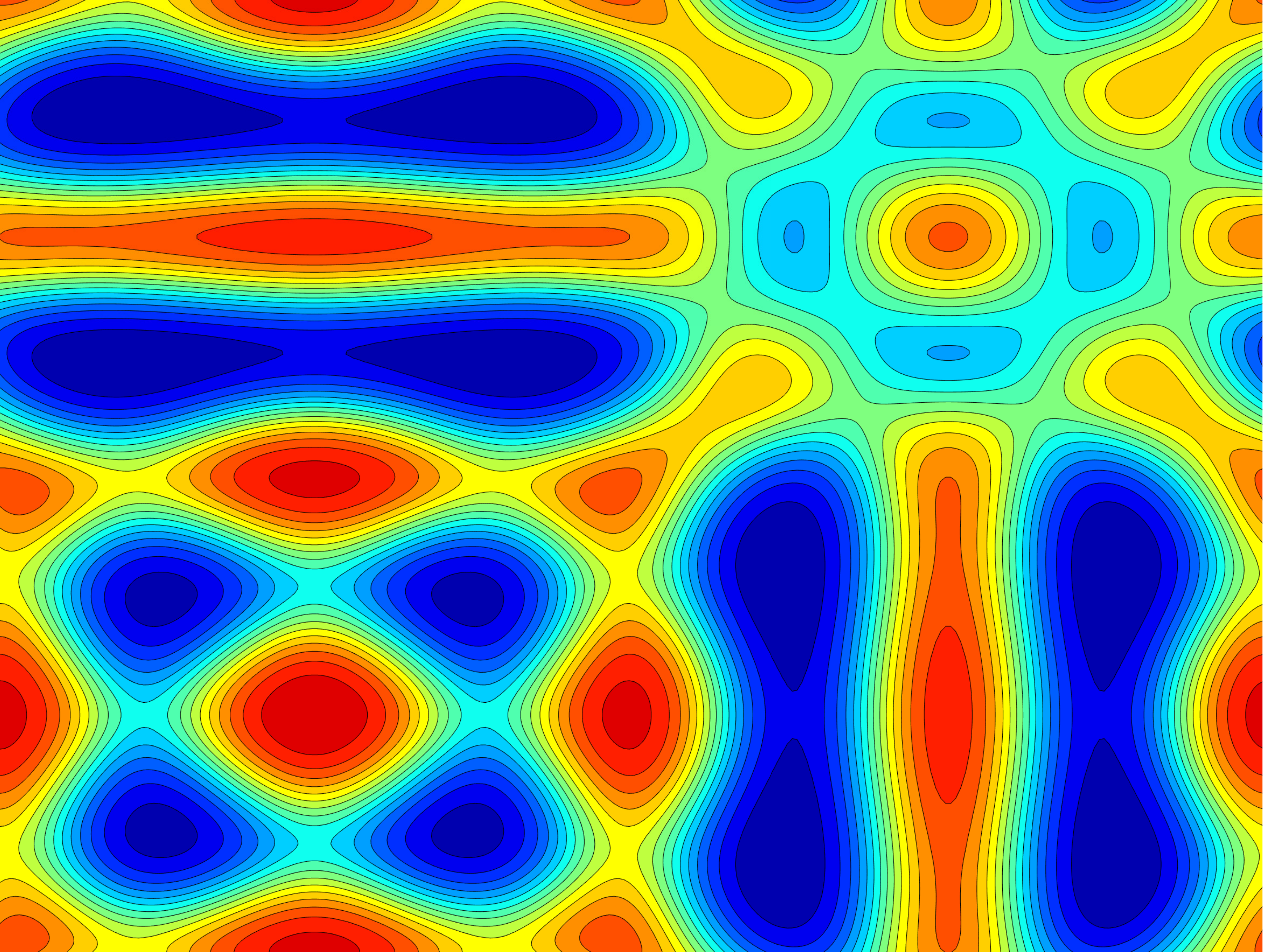}
			\caption*{$t=1, 10$}
		\end{subfigure}
		\begin{subfigure}{0.48\textwidth}
			\includegraphics[height=0.48\textwidth,width=0.48\textwidth]{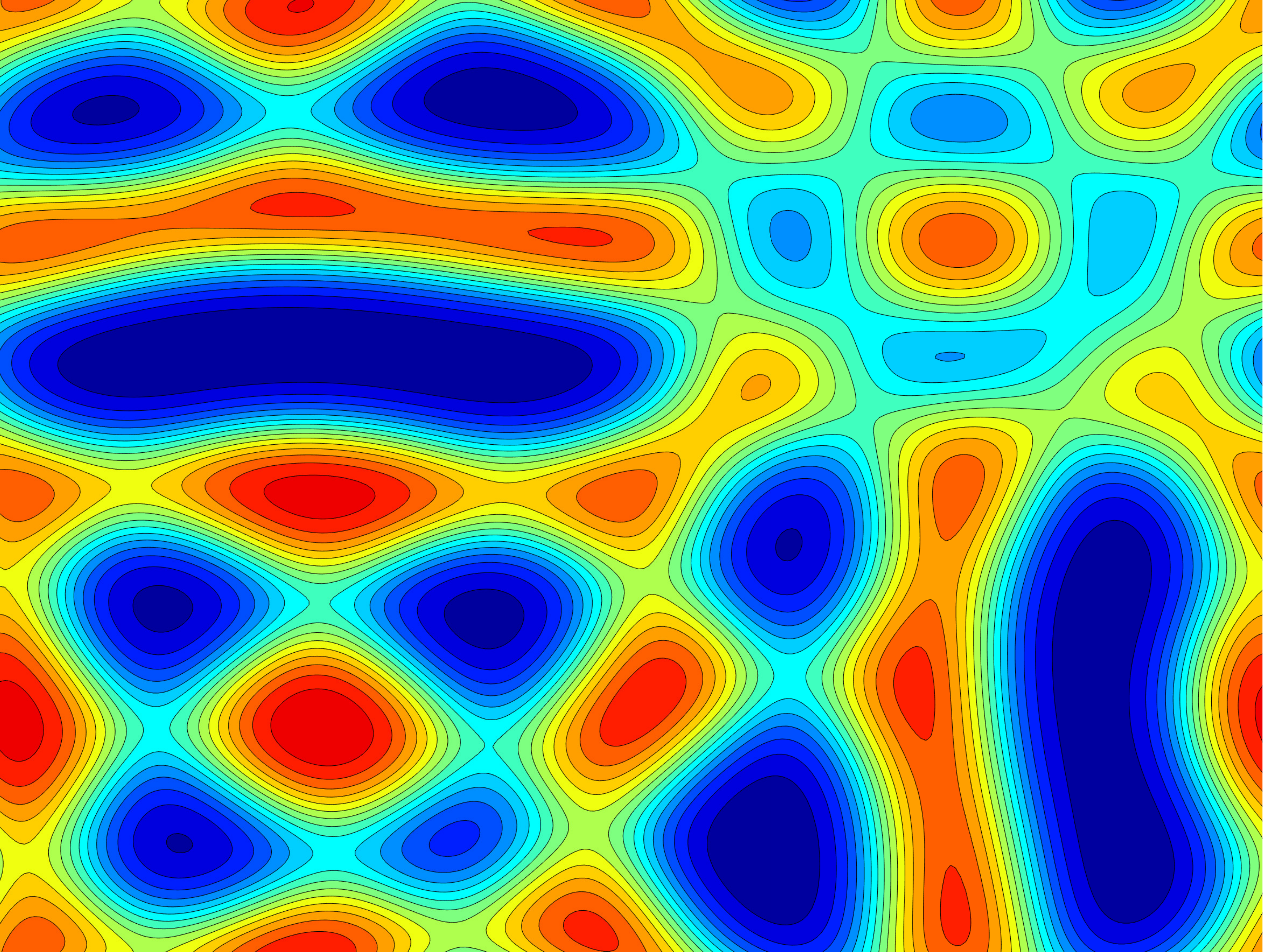} 
			\includegraphics[height=0.48\textwidth,width=0.48\textwidth]{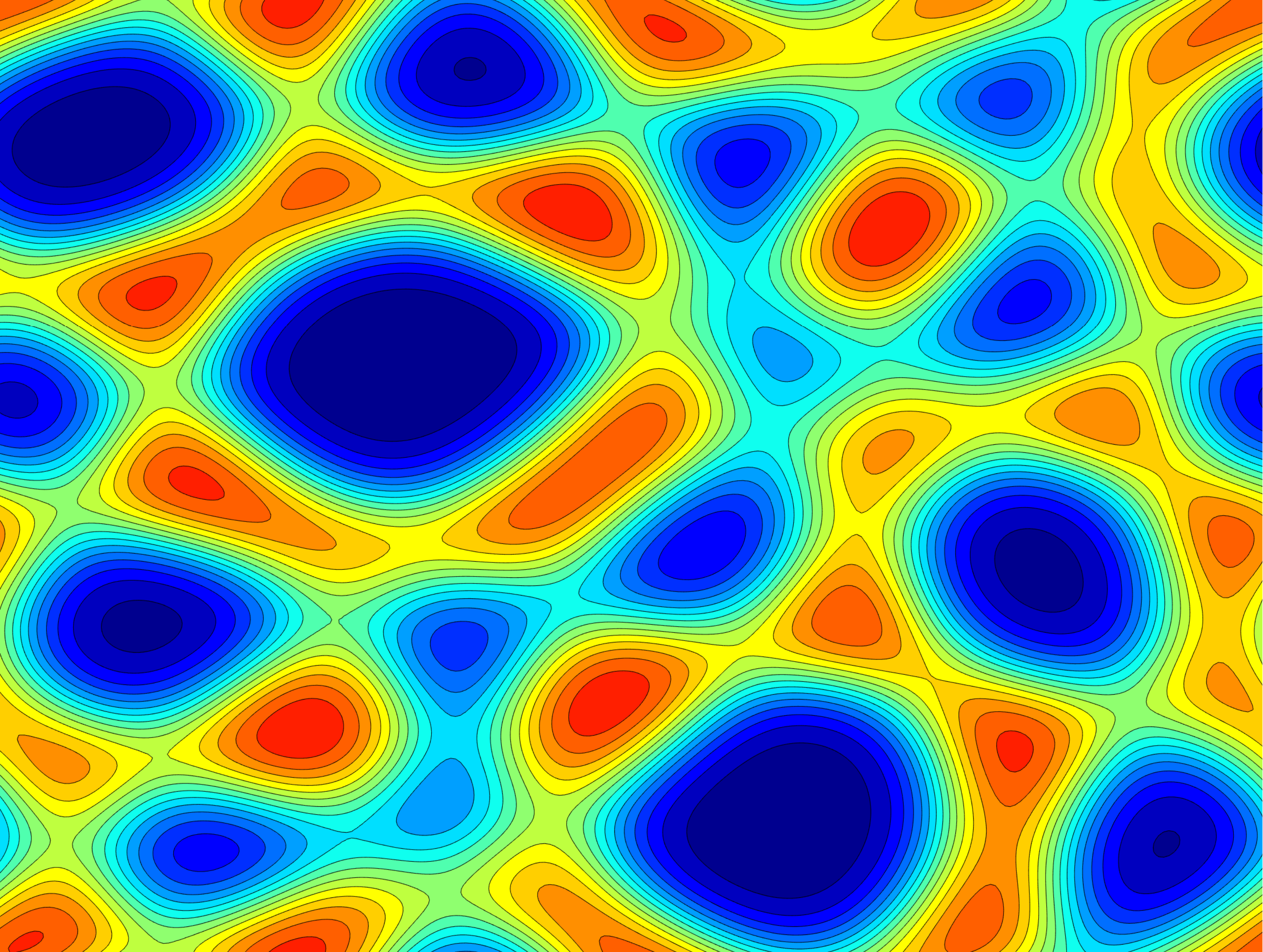}
			\caption*{$t=20, 50$}
		\end{subfigure}
		\begin{subfigure}{0.48\textwidth}
			\includegraphics[height=0.48\textwidth,width=0.48\textwidth]{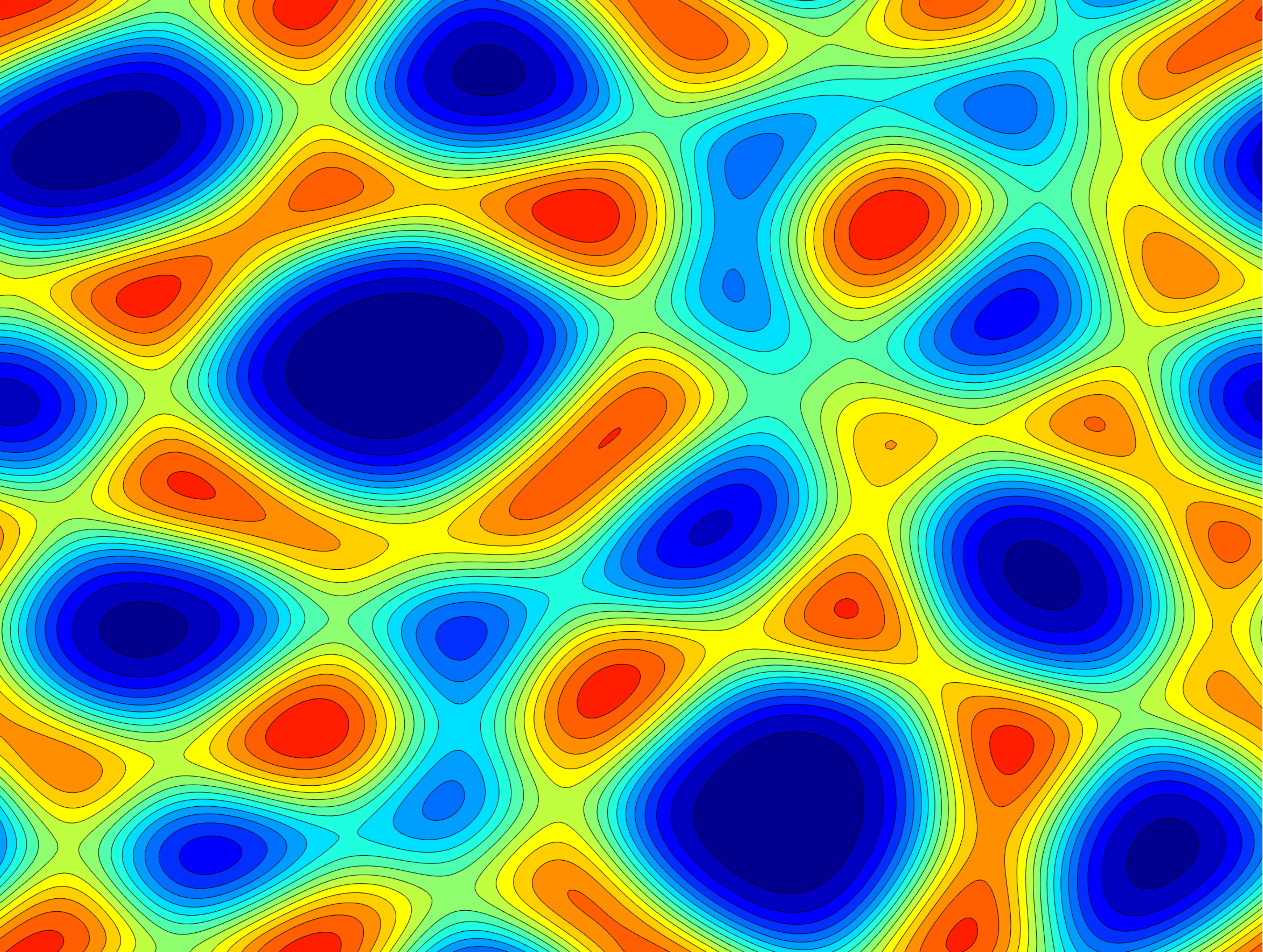}
			\includegraphics[height=0.48\textwidth,width=0.48\textwidth]{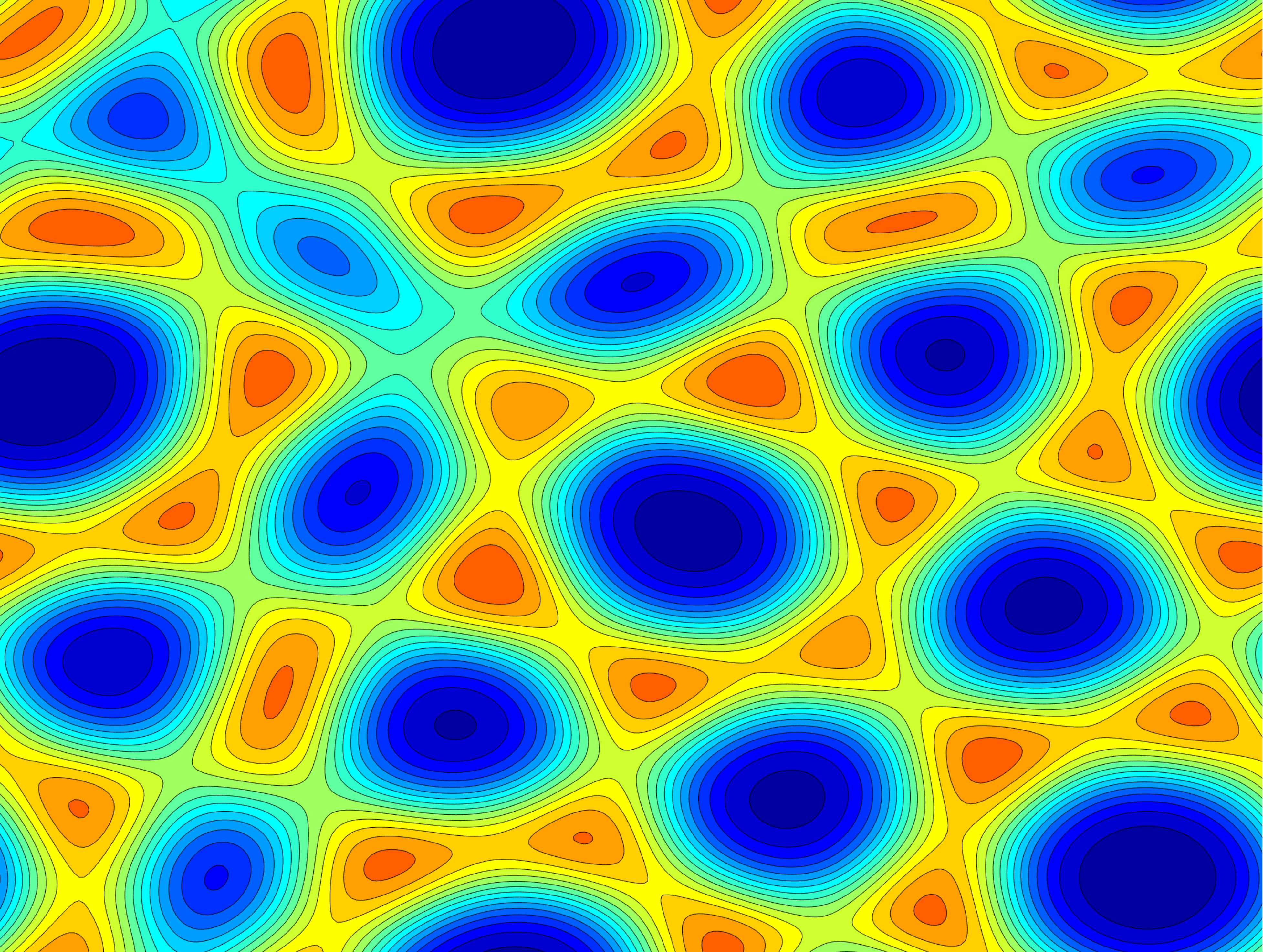}
			\caption*{$t=100,200$}
		\end{subfigure}
		\caption{Time snapshots of the benchmark problem with initial data in \eqref{eqn:init} at $t=0,0.2, 1, 10, 20, 50, 100 ~ \text{and}~200$. The parameters are
			$\varepsilon = 0.18$, $\Omega=(0,6.4)^2$, $A=1.0$, $\eta=1.0$,  $s=1\times 10^{-4}$ and $h={6.4}/{256}$. The numerical results are consistent with earlier work on this topic in \cite{Christlieb14high,jones2013development}.}
		\label{fig:long-time-fch}
	\end{center}
\end{figure}

\subsection{Spinodal decomposition, energy dissipation and mass conservation }
In the second test, we simulate the spinodal decomposition, energy-dissipation and mass-conservation. We  start with the following random initial
condition:
\begin{eqnarray}
\phi(x,y,0)=0.5+0.05(2r-1) , 
\label{eqn:init2}
\end{eqnarray}
where $r$ are the real random numbers in $(0,1)$. The rest of parameters are given by $L_x=L_y=12.8$, $\varepsilon=0.1$, $A=1.0$, $\eta=1.0$, $s=1\times 10^{-4}$ and $h={12.8}/{256}$. The snapshots of spinodal decomposition with initial data in \eqref{eqn:init2} can be found in Fig.~\ref{fig:long-time-micelles}. This experiment also simulates the amphiphilic di-block co-polymer mixtures of polyethylene. The numerical results are consistent with chemical experiments on this topic in \cite{jain2004consequences}. Fig.~\ref{fig:zoonbox} indicates that the simulation has captured all the structural elements with hyperbolic (saddle) surfaces identified in this work, such as short cylinders with one and two beads, cylinder undulation, Y-junction and bilayer-cylinder junction can be found in zoom boxes.

The evolutions of discrete energy and mass for the simulation depicted in Fig.~\ref{fig:long-time-micelles} are presented in Fig.~\ref{fig:energy-mass}. The evolution of discrete energy in Fig.~\ref{fig:energy-mass} demonstrates the energy dissipation property, and the evolution of discrete mass clearly indicates the mass conservation property.    

\begin{figure}[h]
	\begin{center}
		\begin{subfigure}{0.48\textwidth}
			\includegraphics[height=0.48\textwidth,width=0.48\textwidth]{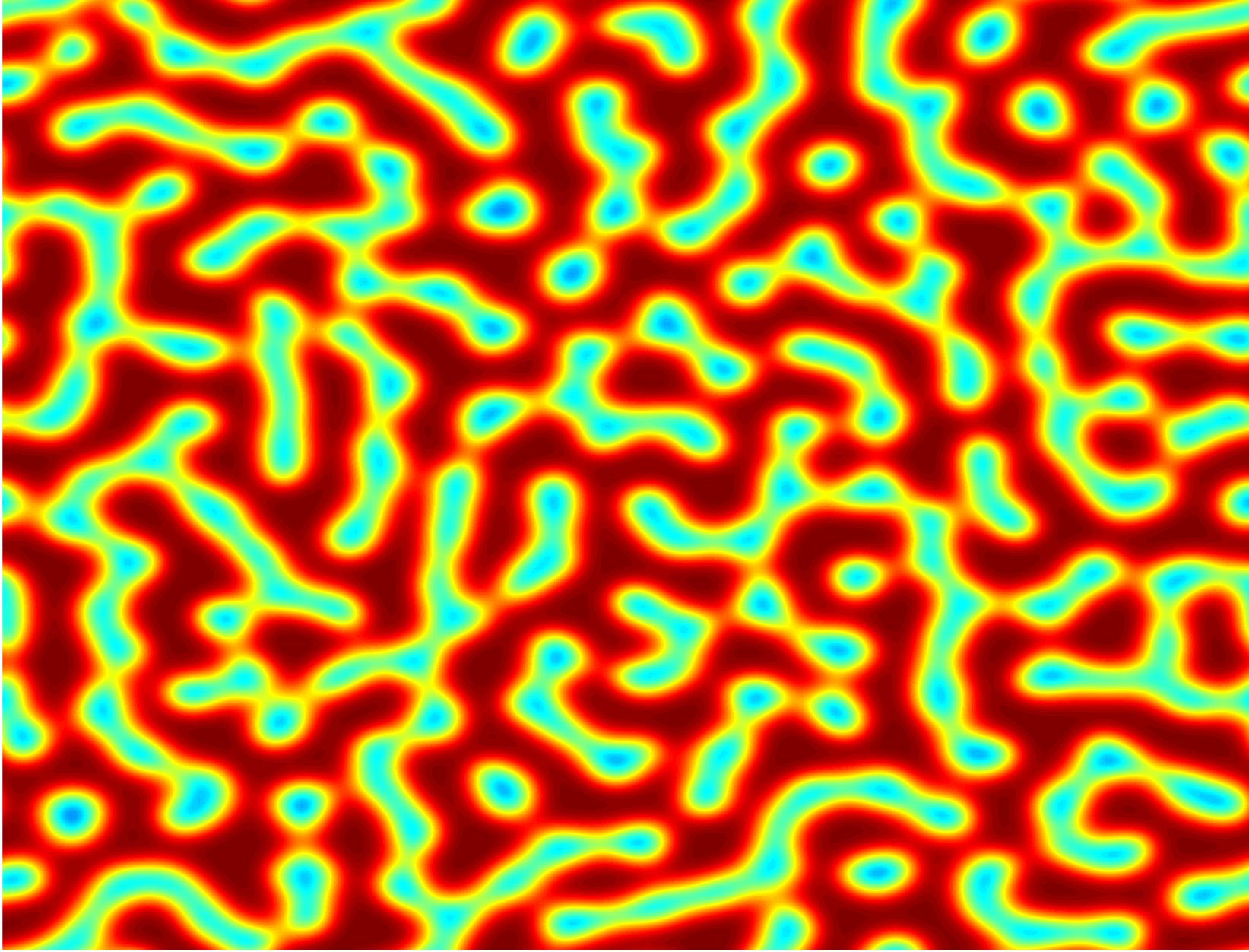} 
			\includegraphics[height=0.48\textwidth,width=0.48\textwidth]{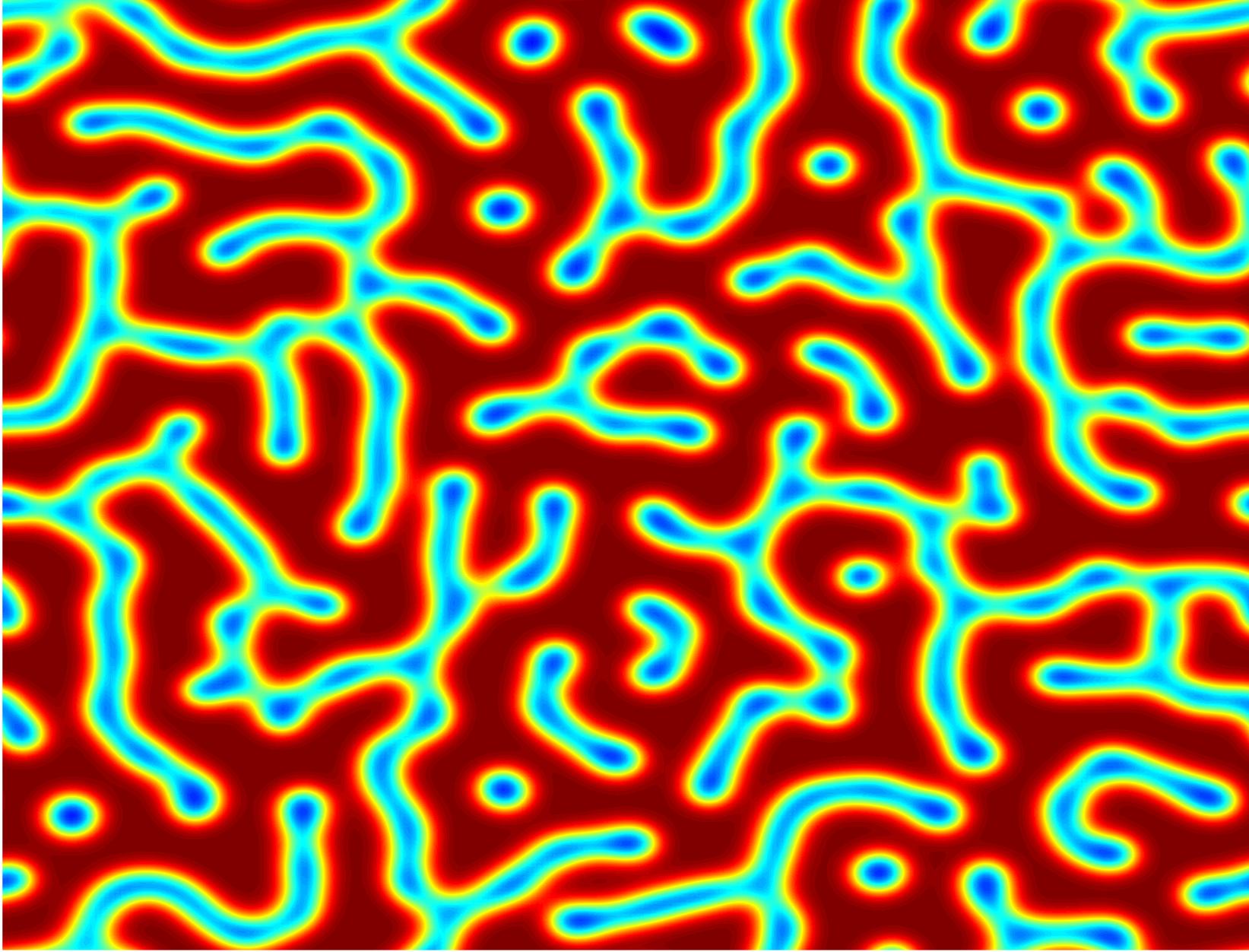} 
			\caption*{$t=0.01, 0.05$}
		\end{subfigure}
		\begin{subfigure}{0.48\textwidth}
			\includegraphics[height=0.48\textwidth,width=0.48\textwidth]{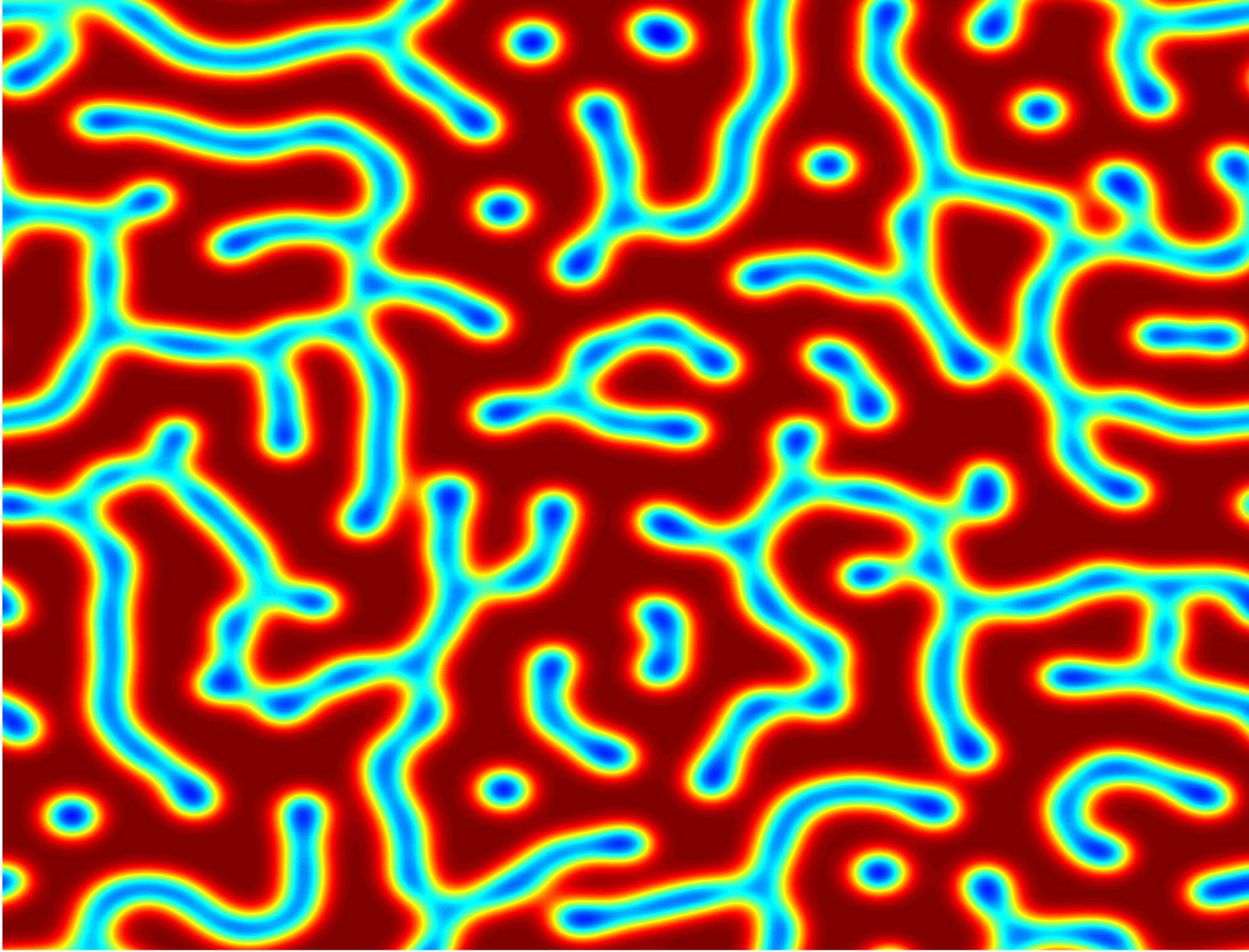} 
			\includegraphics[height=0.48\textwidth,width=0.48\textwidth]{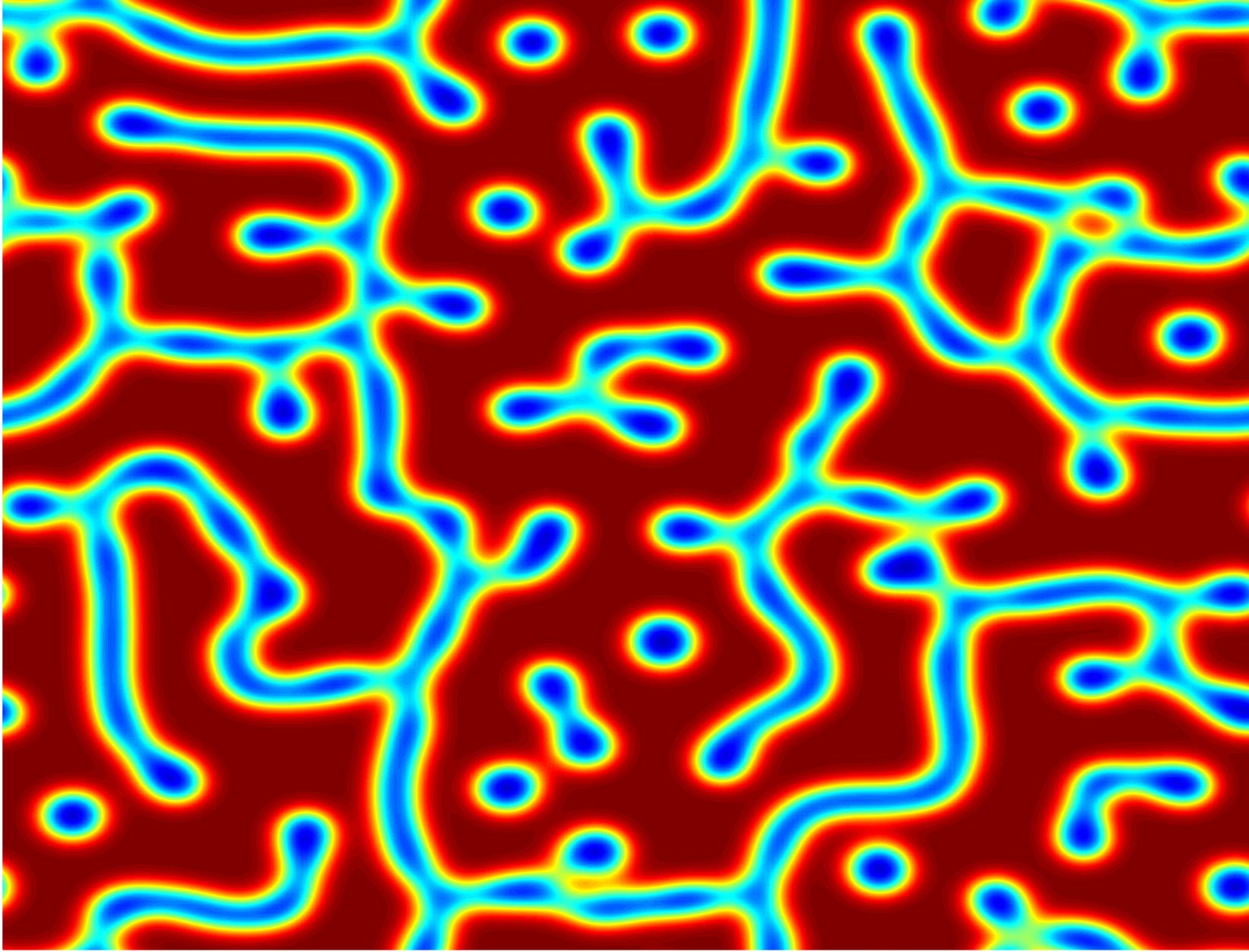}
			\caption*{$t=0.1, 0.5$}
		\end{subfigure}
		\begin{subfigure}{0.48\textwidth}
			\includegraphics[height=0.48\textwidth,width=0.48\textwidth]{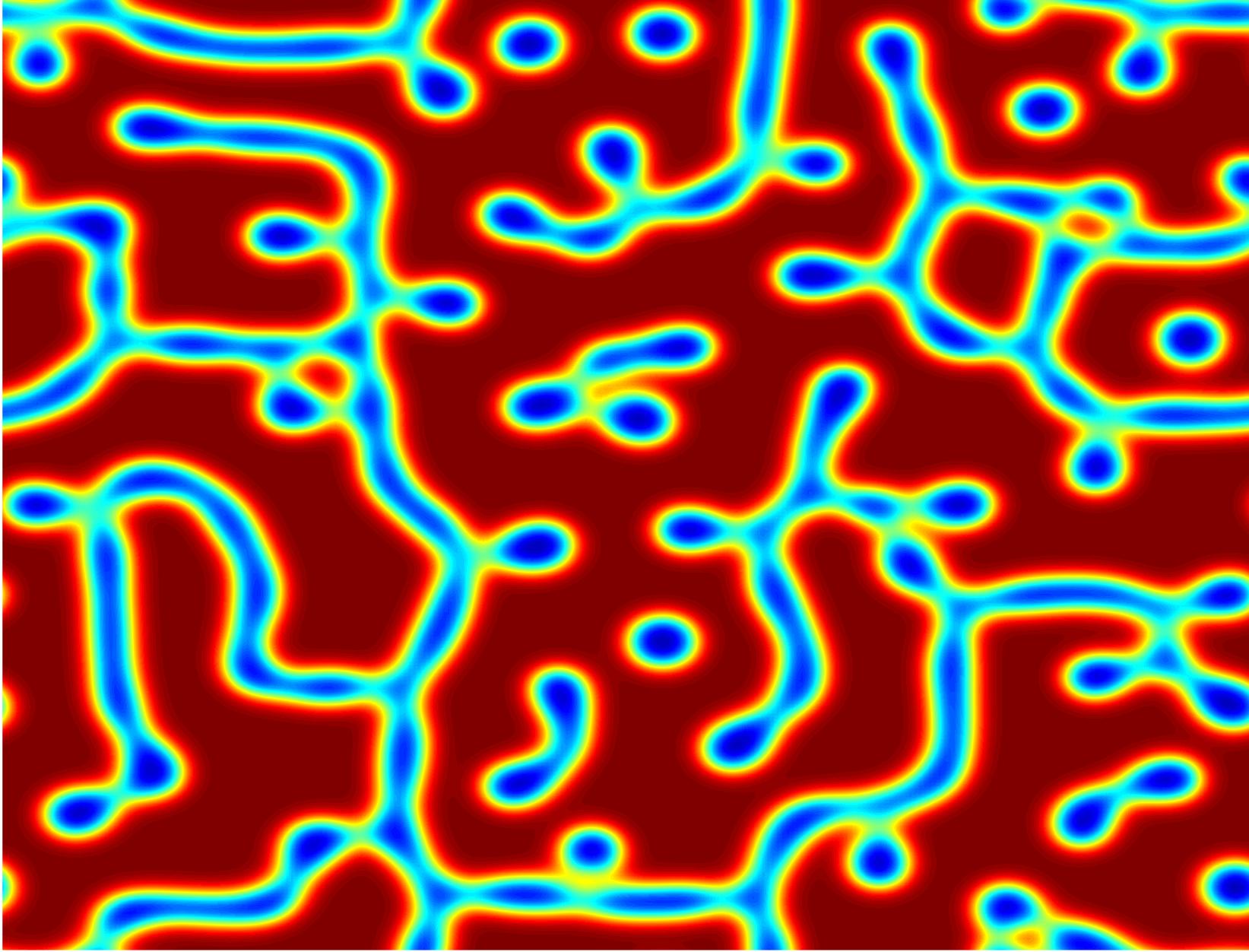} 
			\includegraphics[height=0.48\textwidth,width=0.48\textwidth]{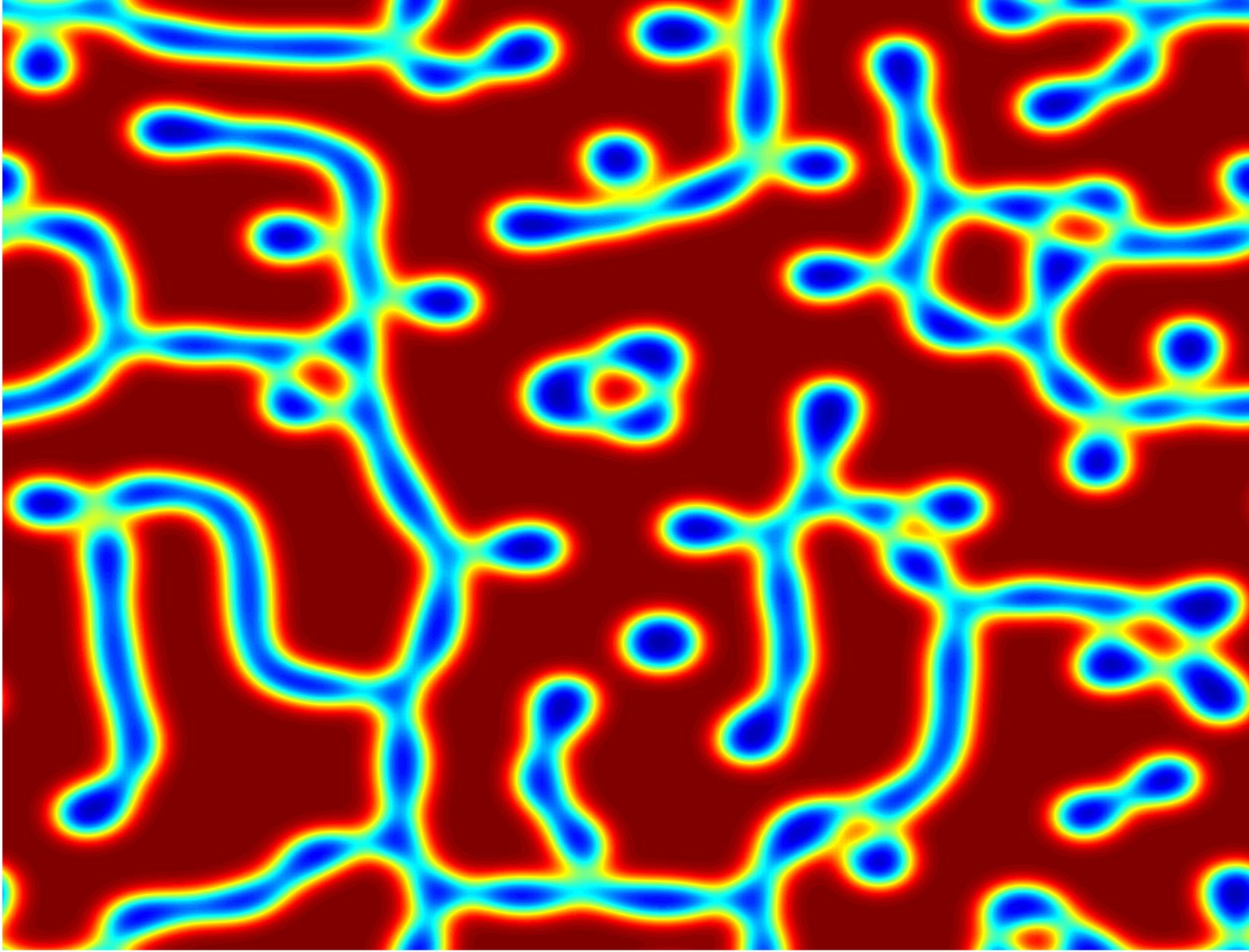}
			\caption*{$t=1, 2$}
		\end{subfigure}
		\begin{subfigure}{0.48\textwidth}
			\includegraphics[height=0.48\textwidth,width=0.48\textwidth]{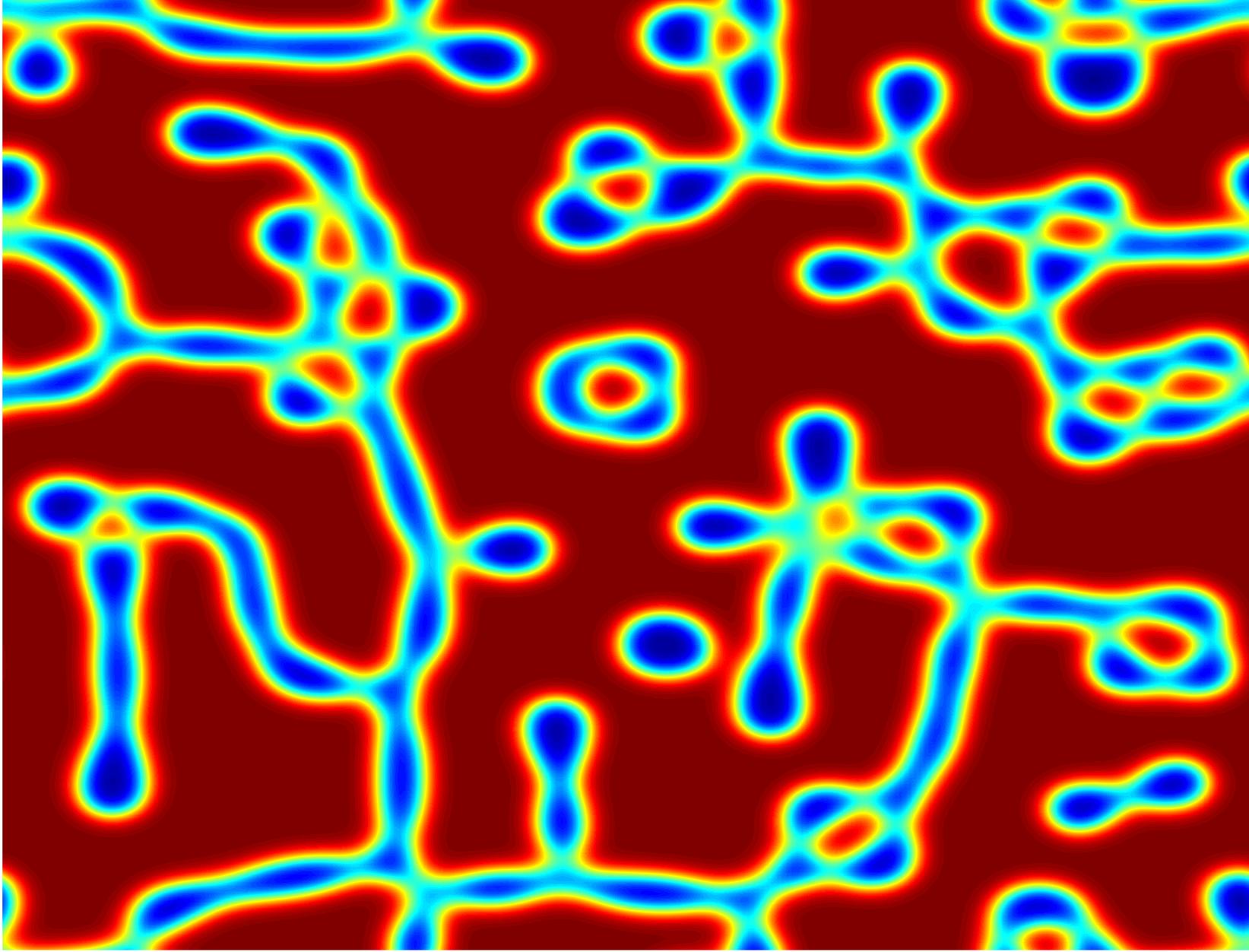}
			\includegraphics[height=0.48\textwidth,width=0.48\textwidth]{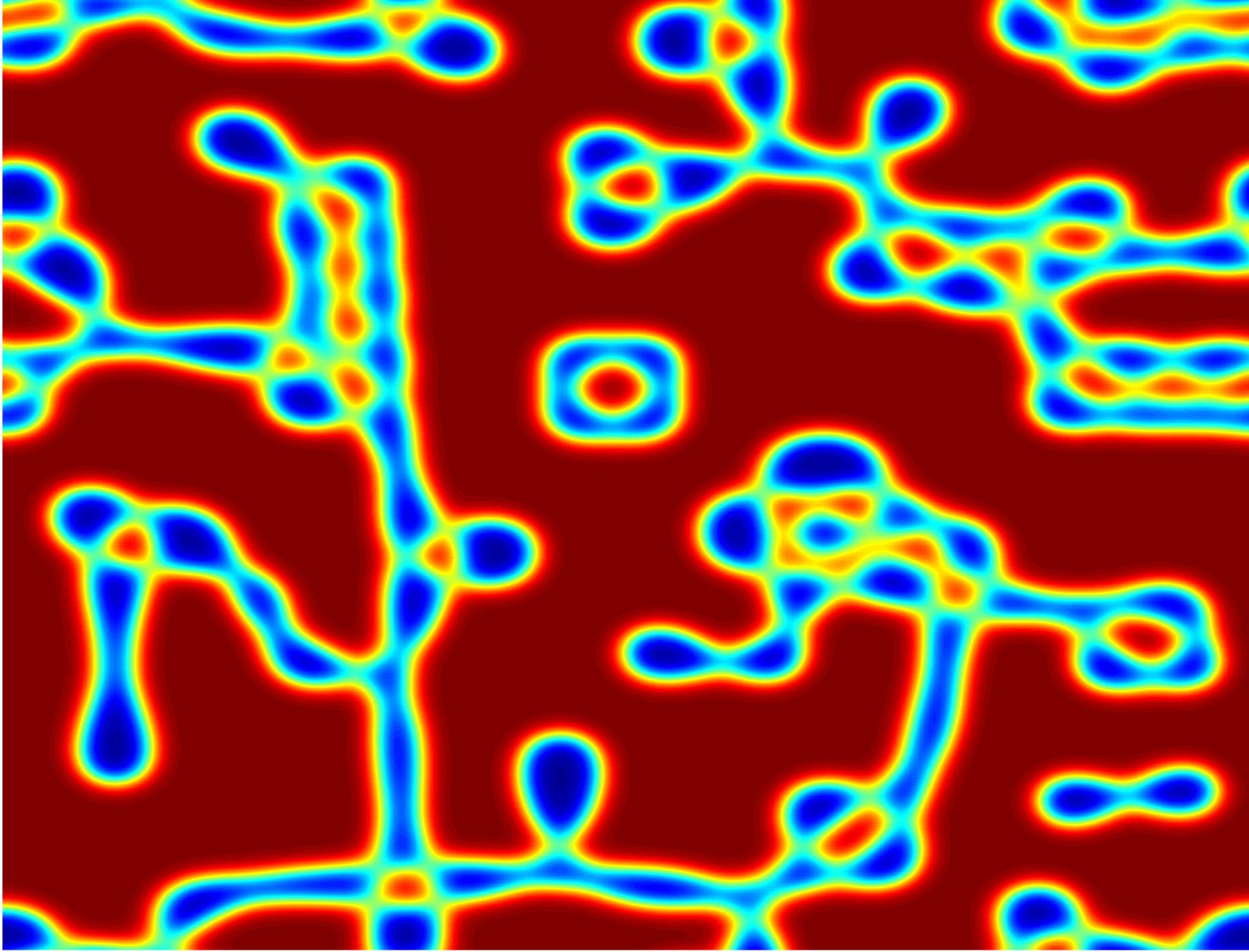}
			\caption*{$t=5,10$}
		\end{subfigure}
		\caption{Snapshots of spinodal decomposition with initial data in \eqref{eqn:init2} at $t=0.01,0.05, 0.1, 0.5, 1, 2, 5 ~ \text{and}~10$. The parameters are $\varepsilon = 0.1, \Omega=[12.8]^2$, $A=1.0$, $\eta=1.0$, $s=1\times 10^{-4}$ and $h={12.8}/{256}$.}
		\label{fig:long-time-micelles}
	\end{center}
\end{figure}

\begin{figure}[ht]\centering
\begin{tikzpicture}[zoomboxarray]
    \node [image node] { \includegraphics[width=0.45\textwidth,height=0.45\textwidth]{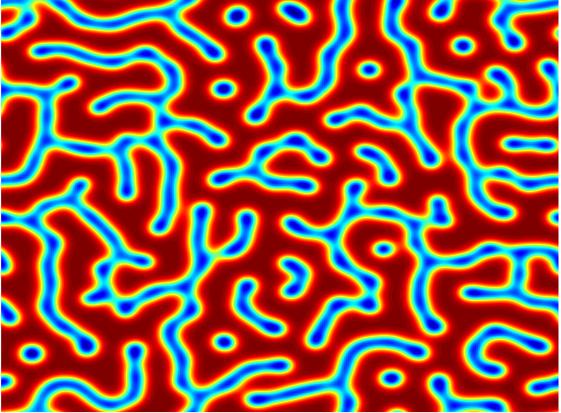} };
    \zoombox[color code=orange,magnification=3.0]{0.45,0.6}
    \zoombox[color code=blue,magnification=2.2]{0.89,0.35}
    \zoombox[color code=yellow,magnification=3.3]{0.45,0.26}
    \zoombox[color code=red,magnification=2.5]{0.89,0.15}      
\end{tikzpicture}
\caption{Left: Snapshots of spinodal decomposition at $t=0.05$. Right: Zoom boxes. Yellow box: Short cylinders with an undulation; Red box: Short cylinders with two undulations; Blue box: Bilayer- Cylinder junction; Orange box: Y-junction. Those numerical results are consistent with chemical experiments on this topic in \cite{jain2004consequences}.}\label{fig:zoonbox}
\end{figure}

\begin{figure}[htp]
\centering
\includegraphics[width=0.48\textwidth]{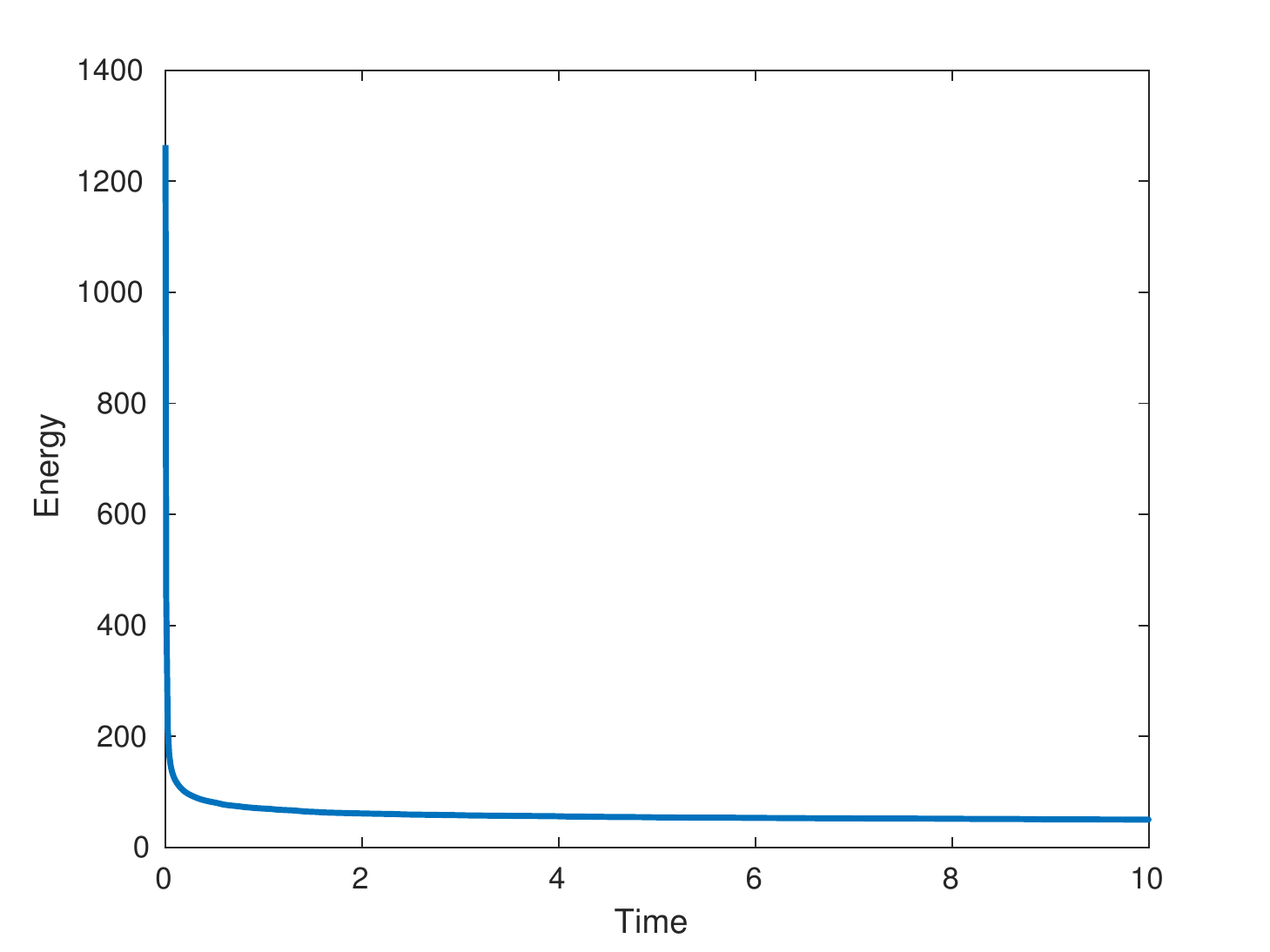}
\includegraphics[width=0.48\textwidth]{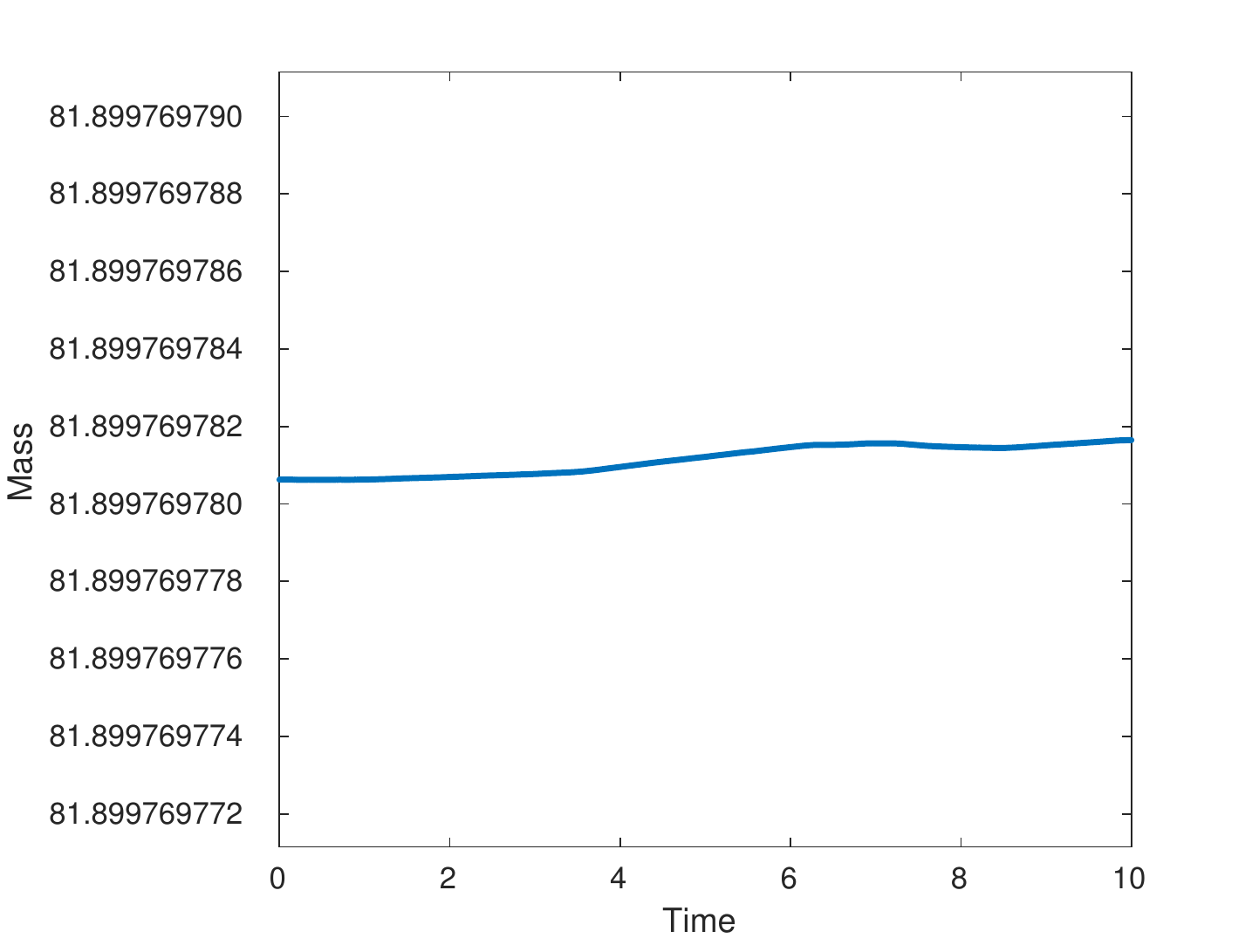}\hfill
\caption{The evolutions of discrete energy and mass for the simulation depicted in Fig.~\ref{fig:long-time-micelles}. Left: Energy Dissipation; Right: Mass Conservation. }
\label{fig:energy-mass}
\end{figure}

\section{Conclusion}\label{sec:conclusion}
We propose and analyze an efficient numerical scheme for solving the FCH equation. Both the unique solvability and unconditional energy stability have been theoretically justified. Based on the global in time $H_{\rm per}^2$ stability of the numerical scheme, we present a rigorous convergence analysis. An efficient PSD method \cite{feng2016preconditioned} is applied to solve the nonlinear system. Various numerical results are also presented, including the first order in time accuracy test,  energy-dissipation, mass-conservation test and the micelle network structures simulation.    
  
\section{Acknowledgements}
JSL acknowledges partial support from NSF-CHE 1035218, NSF-DMR 1105409 and NSF-DMS 1217273. CW acknowledges partial support from NSF-DMS 1418689. SMW acknowledges partial support from NSF-DMS 1418692. 
\bibliographystyle{siam}
\bibliography{fch}
\end{document}